\documentclass{siamltex}

\usepackage{boldfonts}
\usepackage{amsfonts}
\usepackage{amsmath}
\usepackage{amssymb}
\usepackage{array}
\usepackage{graphicx}
\usepackage{color}
\usepackage{stmaryrd}
\usepackage{xifthen}
\usepackage[normalem]{ulem}

\newcommand{\Deform}{\operatorname{Def}}
\newcommand{\Hessian}{\operatorname{Hess}}

\newcommand{\bbR}{\mathbb R}

\newcommand{\divergence}{\textrm{div}}
\newcommand{\T}{\mathcal{T}}
\newcommand{\piola}{\mathcal{P}}
\newcommand{\V}{\mathbb{V}}

\newcommand{\bMT}{\overline{\mathcal{T}}}
\newcommand{\MT}{\mathcal{T}}

\newcommand{\barT}{\overline{T}}
\newcommand{\barx}{\overline{\bx}}

\newcommand{\barq}{\overline{\bf q}}

\newcommand{\bare}{\overline{e}}

\newcommand{\barbn}{\overline{\bf n}}

\newcommand{\barv}{\overline{v}}
\newcommand{\VT}[1][]{ {\bf VT} ({\ifthenelse{\equal{#1}{}}{\gamma}{#1}}) }
\newcommand{\HT}[1][]{ {\bf HT}^{\ifthenelse{\equal{#1}{}}{1}{#1}}(\gamma) }
\newcommand{\HTS}[1][]{ {\bf HT}^{\ifthenelse{\equal{#1}{}}{1}{#1}}(\mathcal S) }
\newcommand{\HTloc}[1]{ {\bf HT}^1(#1) }

\newcommand{\Y}{\mathbb{Y}}

\definecolor{cornellred}{rgb}{0.7, 0.11, 0.11}

\newcommand{\ARD}[1]{\textcolor{black}{{#1}}}

\DeclareFontFamily{U}{matha}{\hyphenchar\font45}
\DeclareFontShape{U}{matha}{m}{n}{
      <5> <6> <7> <8> <9> <10> gen * matha
      <10.95> matha10 <12> <14.4> <17.28> <20.74> <24.88> matha12
      }{}
\DeclareSymbolFont{matha}{U}{matha}{m}{n}
\DeclareFontSubstitution{U}{matha}{m}{n}

\DeclareFontFamily{U}{mathx}{\hyphenchar\font45}
\DeclareFontShape{U}{mathx}{m}{n}{
      <5> <6> <7> <8> <9> <10>
      <10.95> <12> <14.4> <17.28> <20.74> <24.88>
      mathx10
      }{}
\DeclareSymbolFont{mathx}{U}{mathx}{m}{n}
\DeclareFontSubstitution{U}{mathx}{m}{n}

\DeclareMathDelimiter{\vvvert}{0}{matha}{"7E}{mathx}{"17}

\title{A divergence-conforming finite element method for the surface Stokes equation}

\author{
  Andrea Bonito\thanks{Department of Mathematics, Texas A\&M University, College Station TX, 77843; email: {\tt bonito@math.tamu.edu}. 
  Partially supported by NSF Grant DMS-1817691.
}
\and
Alan Demlow\thanks{Department of Mathematics, Texas A\&M University, College Station TX, 77843; email: {\tt demlow@math.tamu.edu}.
Partially supported by NSF Grant DMS-1720369.
}
\and
Martin Licht\thanks{Department of Mathematics, University of California-San Diego; email: {\tt mlicht@ucsd.edu}.  }
}

\begin{document}

\maketitle

\begin{abstract}  The Stokes equation posed on surfaces is important in some physical models, 
but its numerical solution poses several challenges not encountered in the corresponding Euclidean setting.
These include the fact that the velocity vector should be tangent to the given surface
and the possible presence of degenerate modes (Killing fields) in the solution.
We analyze a surface finite element method which provides solutions to these challenges.
We consider an interior penalty method based on the well-known Brezzi-Douglas-Marini $H({\rm div})$-conforming finite element space.
The resulting spaces are tangential to the surface, but require penalization of jumps across element interfaces in order to weakly maintain $H^1$ conformity of the velocity field.  
In addition our method exactly satisfies the incompressibility constraint in the surface Stokes problem.  
Secondly, we give a method which robustly filters Killing fields out of the solution.  
This problem is complicated by the fact that the dimension of the space of Killing fields may change with small perturbations of the surface.  
We first approximate the Killing fields via a Stokes eigenvalue problem and then give a method 
which is asymptotically guaranteed to correctly exclude them from the solution.  
The properties of our method are rigorously established via an error analysis and illustrated via numerical experiments.  
\end{abstract}


\begin{AM} 65N12, 65N15, 65N25, 65N30
\end{AM}

\section{Introduction}
\label{sec:introduction}
We consider the Stokes system on a closed $2$-dimensional surface $\gamma \subset \bbR^{3}$; 
extension to other space dimensions is mostly immediate, 
but we consider the physically most relevant case for the sake of concreteness.  
We assume throughout that $\gamma$ is of class $C^4$. 
Given a tangential forcing term $\bef$, 
the Stokes problem is then to find a divergence-free vector field $\bu$ also tangential to $\gamma$ and pressure $p$ such that 
\begin{align}
    \label{strongform}
    -2\Pi \divergence_\gamma {\rm Def}_\gamma \bu + (\nabla_\gamma p)^\top =\bef \hbox{ on } \gamma.
\end{align}
Here $\Pi$ is the projection onto the tangent space, $\nabla_\gamma$ is the tangential scalar gradient, 
and the tensor ${\rm Def}_\gamma \bu$ is the symmetric tangential gradient of the vector field $\bu$.  
We remark that the choices which lead to this particular form of the surface Stokes system are not always immediately clear.  
For example, the operator $-\Pi \divergence_\gamma {\rm Def}_\gamma$ here is the counterpart of the row wise (vector) Laplacian in the standard Euclidean Stokes system, 
and there are several possible counterparts.  We follow \cite{JOR18} in our definitions and refer to that work for a more in-depth discussion.

The weak form of \eqref{strongform} seeks a pair $(\bu, p) \in \HT \times L_{2,\#}(\gamma)$ such that
\begin{align}
    \label{weakform}
    \begin{aligned}
        2\int_\gamma {\rm Def}_\gamma \bu : {\rm Def}_\gamma \bv - \int_\gamma p \, \divergence_\gamma \bv & = \int_\gamma {\bf f} \cdot \bv, ~~\bv \in \HT,
        \\ \int_\gamma \divergence_\gamma \bu \, q& = 0, ~~q \in L_{2,\#}(\gamma).
    \end{aligned}
\end{align}
Here $L_{2,\#}(\gamma)$ is the subspace of $L_2(\gamma)$ with vanishing mean value, 
and $\HT$ consists of vector fields in $H^1(\gamma)^3$ which are tangent to $\gamma$ a.e.  

In this work we confront two challenges that arise when discretizing \eqref{weakform}.  
First, construction of conforming finite element subspaces of $\HT$ is not straightforward, and in fact has not been accomplished to date.  The reason is that finite element vector fields lying in such a discrete space must simultaneously be continuous across element interfaces (due to membership in $H^1(\gamma)^3$) and tangent to the surface $\gamma$.  
Recent works on finite element methods for \eqref{weakform} and similar problems have approached this problem 
by enforcing the tangential constraint only weakly either by penalization \cite{OQRY18} or by a Lagrange multiplier approach \cite{GJOR18} while preserving membership in $H^1$.  
Our approach is complementary in that we enforce the tangent constraint exactly while enforcing $H^1$ continuity weakly via an interior penalty approach; \ARD{cf. \cite{LLS20}}.  We use $H(\divergence_\gamma)$-conforming finite element spaces, and our method is the surface counterpart to well-established methods for the Euclidean Stokes problem \cite{CKS05,CKS07}.  
Our method has the advantage of being divergence conforming, that is, the constraint $\divergence_\gamma \bu = 0$ is enforced exactly.
In this work we focus on the lowest-order Brezzi-Douglas-Marini (BDM) space in order to simplify our error analysis, but our approach could also be used with other $H(\divergence_\gamma)$-conforming spaces.  \ARD{In particular, our basic algorithm applies without substantial modification to higher-order BDM spaces as in \cite{LLS20}.}  

The second challenge which we confront here is possible ill-posedness of \eqref{strongform} and \eqref{weakform} 
in the presence of rigid motions \ARD{and other continuous intrinsic isometries of $\gamma$; cf. \cite{BBSG10} for a discussion of the relationship between Killing fields and surface properties in an application-based setting.  The corresponding vector fields} are known as \emph{Killing fields} and constitute the (possibly empty) subspace $\mathcal{K}$ of ${\rm HT}^1(\gamma)$
whose members are annihilated by the deformation operator $\Deform_\gamma$.  The dimension of $\mathcal{K}$ is at most three, with equality holding only on the sphere.  \ARD{Because $\mathcal{K}$ has positive dimension only when $\gamma$ possesses a certain structure,} we generically have ${\rm dim}(\mathcal{K}) = 0$.
A well-posed version of \eqref{weakform} may be obtained by requiring that $\bu \perp \mathcal{K}$.  

Similarly filtering out the Killing fields in a computational setting may be a much more subtle problem.
In some cases (such as when $\gamma$ is a sphere), it is not difficult to find the dimension of $\mathcal{K}$ and even a basis for it.
It is also easy to see that Killing fields are eigenmodes corresponding to zero eigenvalues of the eigenvalue problem for the system \eqref{weakform}, so in principle one could compute these eigenfunctions and remove these modes from the solution.  
The difficulty with this approach is that even ${\rm dim}(\mathcal{K})$ may be unknown. 
For instance, it may be the case in computational practice that information about $\gamma$ comes in the form of a triangulated (polyhedral) approximating surface $\Gamma$ which may not inherit the symmetries of the underlying surface $\gamma$.  Thus the set of Killing fields on $\Gamma$ may be trivial even if $\mathcal{K}$ is not.  If Killing fields are present, they would thus correspond to {\it small} eigenvalues of the Stokes system on $\Gamma$ rather than zero eigenvalues.  
Discretization also adds a different complication.  As we describe below, conforming approximations to \eqref{weakform} are unknown, and various penalization techniques used to obtain convergent FEM will also perturb the eigenvalues of the Stokes system.  Thus we must somehow distinguish between small eigenvalues of a discrete system that correspond to actual zero continuous eigenvalues and actual Killing fields,
and those that correspond to near-symmetries of $\gamma$ and thus not to genuine Killing fields.   We contrast this situation with the role of harmonic forms in the finite element exterior calculus framework, which similarly can be viewed as zero eigenmodes and must be filtered out to obtain well-posed formulations of Hodge-Laplace problems.  However, harmonic forms are also zero discrete eigenmodes and thus easily identified in a discrete setting \cite{AFW10}.  The dimension of the space of harmonic forms is also a topological invariant and so is stable with respect to small geometric perturbations. 

It appears to be very difficult to firmly identity ${\rm dim}(\mathcal{K})$ computationally absent precise knowledge of the \ARD{continuous intrinsic} symmetries of $\gamma$.    We thus pursue a more modest goal, 
which is to obtain an optimally convergent approximation to $\bu-P_{\mathcal{K}}\bu$ in the $H^1$ and $L_2$ norms.
The main tool we use is the following perturbation of \eqref{weakform}:
Given $\varepsilon >0$, find a pair $(\bu^\varepsilon, p^\varepsilon) \in \HT \times L_{2,\#}(\gamma)$ such that
\begin{align}
    \label{pert_weakform}
    \begin{aligned}
        2\int_\gamma {\rm Def}_\gamma \bu^\varepsilon : {\rm Def}_\gamma \bv + \varepsilon \int_\gamma \bu^\varepsilon \cdot \bv - \int_\gamma p^\varepsilon \, \divergence_\gamma \bv & = \int_\gamma {\bf f} \cdot \bv, ~~\bv \in \HT,
        \\ \int_\gamma \divergence_\gamma \bu^\varepsilon \, q& = 0, ~~q \in L_{2,\#}(\gamma).
    \end{aligned}
\end{align}
We establish below that $\bu^\varepsilon \perp_{L_2} \mathcal{K}$ and that $\bu-\bu^\varepsilon=O(\varepsilon)$ in the $L_2$ and $H^1$ norms.  

Our finite element method is based on the perturbed problem \eqref{pert_weakform}, 
and its convergence properties depend on the choice of $\varepsilon$.  
Roughly speaking, a discrete solution $\bU^\varepsilon$ corresponding to a finite element mesh of size $h$ satisfies 
\begin{align*}
    \vvvert \bu-\bU^\varepsilon \vvvert_{1,h,\varepsilon} \lesssim (h+\varepsilon) \|{\bf f}\|_{L_2(\gamma)},
\end{align*}
where $\vvvert \cdot \vvvert_{1,h,\varepsilon}$ is a discontinuous Galerkin energy norm for the elliptic portion of the problem \eqref{pert_weakform}.  
Because the lowest-order BDM space approximates $\bu$ using piecewise linears, 
this estimate is optimal if $\epsilon \le h$, and in particular if either $\varepsilon=h$ or $\varepsilon = h^2$.
In addition,
\begin{align}
    \label{first_L2}
    \|\bu-\bU^\varepsilon\|_{L_2(\gamma)} \lesssim (h^2+\varepsilon + \frac{h^2}{\varepsilon})\|{\bf f}\|_{L_2(\gamma)},
    \\
    \label{second_L2}
    \|\bu-(\bU^\varepsilon-P_{\mathcal{K}} \bU^\varepsilon)\|_{L_2(\gamma)} \lesssim (\varepsilon + h^2)\|{\bf f}\|_{L_2(\gamma)}.
\end{align}
Choosing $\varepsilon = h$ yields a suboptimal order of convergence in both of these estimates.  
However, \eqref{first_L2} also indicates that obtaining a convergent approximation in $L_2$ does {\it not} require any knowledge about $\mathcal{K}$.  
On the other hand, taking $\varepsilon = h^2$ yields an $O(1)$ error in \eqref{first_L2} if Killing fields are not filtered out,
but an optimal $O(h^2)$ convergence rate in \eqref{second_L2} if they are.
Thus the choice of $\varepsilon$ allows for a tradeoff between optimal convergence rates in $L_2$ and robustness with respect to the presence of Killing fields.  Instead of insisting on filtering out Killing fields, it may also be reasonable to accept a solution that is accurate up to a Killing field.  Estimate \eqref{second_L2} indicates that $\bU^{h^2}$ converges optimally to $\bu$ in this sense.  
 
We also address cases where $\mathcal{K}$ is not known explicitly and thus must be discretely approximated.   
We first establish optimal convergence in $L_2$ of discrete eigenfunctions of the Stokes eigenvalue problem corresponding to \eqref{weakform}.
If ${\rm dim}(\mathcal{K})$ is known, we may substitute the projection $P_{\mathcal{K}_h}$ onto the corrresponding discrete eigenspace for the projection $P_{\mathcal{K}}$ in the estimate \eqref{second_L2} and maintain optimal convergence.  
If ${\rm dim}(\mathcal{K})$ is not known, it is not possible at any given mesh level to determine whether a given space of discrete eigenfunctions corresponds to $\mathcal{K}$. 
In order to overcome this problem, we choose a candidate set of discrete Killing fields as the discrete eigenspace $\mathcal{K}_h$ which minimizes $\|\bU^{h^\alpha}-(\bU^{h^2}-P_{\mathcal{K}_h} \bU^{h^2})\|_{L_2(\gamma)}$ for some $\alpha \in [ 1, 2 )$.  
\ARD{Implementation of this condition requires only checking a simple inequality involving the mesh size and discrete eigenvalues.}  
Below we show that for $h$ sufficiently small this choice yields a space of discrete Killing fields which converges optimally to $\mathcal{K}$,
and thus $\|\bu-(\bU^{h^2}-P_{\mathcal{K}_h} \bU^{h^2})\|_{L_2(\gamma)} \lesssim h^2$.  
On the other hand, for $h$ in the preasymptotic range we use \eqref{first_L2} to find that
$\|\bu-(\bU^{h^2}-P_{\mathcal{K}_h} \bU^{h^2})\|_{L_2(\gamma)} \lesssim h^{2-\alpha}$.  
Thus even if we do not correctly determine the space of discrete Killing fields, 
we nonetheless obtain a reasonable approximation via this algorithm, 
and asymptotically we are guaranteed to correctly filter out the Killing fields.  
We emphasize that this algorithm does not require any a priori knowledge of $\mathcal{K}$, even its dimension.  
It may seem obvious that we should choose $\alpha=1$ in order to maximize the order of convergence in the preasymptotic range.
As we illustrate via numerical experiments below, however, this intuition is not necessarily correct.   \ARD{Finally, we consider a similar analysis when ${\bf f} \not\perp \mathcal{K}$.}  

A number of recent papers have considered numerical methods for the surface Stokes and related vector Laplace-type problems,
including \cite{FR18, GJOR18, JR20, LLS20, NRV17,OQRY18, ORZ20, OY19, RV17}.  We first mention \cite{OQRY18}, 
where a trace surface FEM for the surface Stokes equation is defined and error analysis is given.  There an $H^1$-conforming method is considered and penalization is used to weakly enforce tangential conformity for the velocity space.
The problem \eqref{pert_weakform} is considered to be the continuous problem instead of \eqref{weakform}.
The ``pure'' Stokes problem with $\varepsilon = 0$ is only allowed if ${\rm dim}(\mathcal{K}) = 0$,
and otherwise it is assumed that $\varepsilon>0$ in order to obtain a well-posed problem independent of the presence of Killing fields.   \ARD{The paper \cite{JR20} explores options for weakly enforcing tangengiality by penalty and Lagrange multiplier methods.  A full error analysis is given, including consideration of ``geometric errors'' due to approximation of the surface $\gamma$ in the numerical method.}  In \cite{Re18} the author gives a stream function formulation of the surface Stokes equations that is well-suited to finite element discretization, \ARD{and a corresponding finite element method and error analysis are given in \cite{BRPP}}.  However, this methodology is limited to simply connected surfaces due to the presence of harmonic forms otherwise.  \ARD{Our work and the recent paper \cite{LLS20} have a number of features that are distinct from these previous works.}  These include tangential conformity of the method and exact enforcement of the divergence-free constraint.
In addition, our method places no restrictions on domain geometry,
allowing both for non-simply connected surfaces and robustness in the presence of Killing fields. 
\ARD{The paper \cite{LLS20} also consider tangential- and divergence-conforming FEM for surface Navier-Stokes equations, but focuses on computational and algorithmic aspects.  The lowest-order interior penalty method considered there is essentially the same as our, and higher-order and hybridizable discontinuous Galerkin (HDG) versions are also presented.  A number of illuminating computational experiments are carried out for the stationary and instationary surface (Navier-)Stokes equations, but no analysis is carried out.  In contrast, we give rigorous error analysis, including the first error analysis for a surface Stokes eigenvalue problem and the approximation of Killing fields.}

The rest of the paper is arranged as follows.   In Section~\ref{sec:preliminaries} we discuss analytical preliminaries, and in Section~\ref{sec:todo} we explore the perturbed problem \eqref{pert_weakform}.  In Section~\ref{sec:fem} we define our finite element method and prove basic energy estimates, while in Section~\ref{sec:L2} we prove $L_2$ error estimates.  Section~\ref{sec:killing} contains definitions and discussion of methods for filtering out Killing fields.  In Section~\ref{sec:numerics} we illustrate our results via numerical experiments.  

\section{Preliminaries}
\label{sec:preliminaries}
%
%
%

We use the notation $A \lesssim B$ to denote $ A\leq C B$ for a constant $C$ independent of $A$ and $B$. 
Furthermore, $A\gtrsim B$ indicates $B \lesssim A$ and  we write $A \simeq B$ whenever $A\lesssim B$ and $B \lesssim A$.

\subsection{Implicit surface representation}
\label{sec:surf_rep}

We assume that $\gamma$ is a $C^{4}$ surface, that is, 
the zero level set of a $C^{4}$ function over the ambient space.  \ARD{Our algorithm can be formulated for less regular surfaces, and it the error analysis also may not require that the surface be $C^4$.  This assumption arises because of technical issues related to proving error estimates for surface FEM under optimal assumptions on surface regularity; cf. \cite{bonito2019finite} for a discussion in the context of scalar elliptic problems.  In addition, the Piola transform that we employ below involves multiplication by terms involving derivatives of the transforming function and therefore raises by one differentiability degree the natural assumption on surface regularity.}  We also assume that $\gamma$ is compact and closed. 
Thus $\gamma$ divides the ambient into two sets: the compact \emph{interior domain} and the \emph{exterior domain}. 
There exists an open neighborhood $\mathcal{N}$ of the surface $\gamma$
and a \emph{signed distance function} $d \in C^4(\mathcal{N})$ for which $d < 0$ over the interior domain
and $d > 0$ over the exterior domain and for which $|d(x)|={\rm dist}(x, \gamma)$.

The Hessian $\bH=\Hessian d$ is known as the \emph{Weingarten map}.
Note that $\bH$ has one zero eigenvalue corresponding to the direction $\nabla d$, and for $\bx \in \gamma$ the other eigenvalues $\kappa_i$ ($i=1,2$) of $\bH$ are the principal curvatures of $\gamma$.  These quantities are all defined on a tubular neighborhood of width $1/\|\max_{i} |\kappa_i|\|_{L_\infty(\gamma)}$, and in order to avoid degeneration of constants near the boundary of this neighborhood we assume $\mathcal{N} \subseteq \{\bx \in \mathbb{R}^{3}: |d(\bx)| \le 1/(2 \|\max_{i} |\kappa_i|\|_{L_\infty(\gamma)}) \}$ without loss of generality.

Note that $\nabla d$ is the outward-pointing normal along $\gamma$. 
We define a vector field $\bnu \in C^{3}(\mathcal{N})^{2}$ by setting $\bnu(\bx) = \nabla d(\bx)$ when $\bx \in \gamma$
and taking the extension onto $\mathcal{N}$ that is constant in the normal direction. 
Under those assumptions, the closest point projection 
\begin{equation}\label{e:P}
    \bP : \mathcal{N} \rightarrow \gamma, \quad \bx \mapsto \bx - d( \bx ) \bnu( \bx )
\end{equation}
is uniquely defined. 
Note that $\bnu(\bx)=\bnu(\bP(\bx))$ for $x \in \mathcal{N}$.  
In addition, the matrix field $\Pi = \bI - \bnu \otimes \bnu \in C^{3}(\mathcal{N})^{3 \times 3}$ 
describes the projection onto the tangent plane along $\gamma$.  We refer to \cite{bonito2019finite, Dz88, DD07} for further discussion of these properties.

\subsection{Differential operators and function spaces}  Let $L_{2}(\gamma)$ be the Hilbert space of square-integrable functions over $\gamma$
and denote by $L_{2,\sharp}(\gamma)$ the subspace of $L_{2}(\gamma)$ whose members have vanishing average value.
We let $H^{l}(\gamma)$ denote the Sobolev space of order $l$ over $\gamma$
and write $\| .\|_{H^l(\gamma)}$ for the corresponding norm. 
For sufficiently smooth scalar functions $u:\gamma \rightarrow \mathbb{R}$,
we denote by $\nabla_\gamma u$ the tangential gradient of $u$ on $\gamma$.
If $u$ is defined in a neighborhood of $\gamma$, then we may also write $\nabla_\gamma u=  \nabla u \Pi$.
Here we follow the convention that $\nabla_\gamma u$ is a row vector.  


For a surface $\gamma$ we let $\HT$ denote the closed subspace of $H^1(\gamma)^{3}$ 
whose members are tangent to $\gamma$ almost everywhere.  
We adopt the convention that members of $\HT$ are column vectors.  
In addition, we set $\HT[2] := \HT \cap H^2(\gamma)^{3}$ equipped with the $H^2(\gamma)^{3}$ norm.
The tangential total derivative of a vector function is given by 
\begin{align*}
  \nabla_\gamma : H^1(\gamma)^{3} \rightarrow L^{2}(\gamma)^{3 \times 3}, 
  \quad \bv \mapsto \Pi \nabla \bv \Pi.
\end{align*}
The deformation of a tangential vector field is defined as the symmetric part of the tangential total derivative, 
\begin{align*}
 \Deform_{\gamma} :  H^1(\gamma)^{3} \rightarrow L^{2}(\gamma)^{3 \times 3}, 
 \quad \bv \mapsto \frac 1 2 ( \nabla_\gamma \bv + \nabla_\gamma^T \bv ).  
\end{align*}
The tangential divergence of a tangential vector field is given by 
\begin{align*}
  \divergence_{\gamma} : \HT \rightarrow L^{2}(\gamma),
  \quad \bv \mapsto {\rm tr} \nabla_\gamma \bv,
\end{align*}
while the rowwise divergence of a matrix field is 
\begin{align*}
 \divergence_{\gamma} : H^{1}(\gamma)^{2 \times 2} \rightarrow L^{2}(\gamma)^{3},
 \quad 
 \bA \mapsto \big (\divergence_\gamma ({\bf e}_1^T \bA),\dots, \divergence_\gamma ({\bf e}_{3}^T \bA) \big)^T.
\end{align*}
%
%
Members of the kernel of the deformation operator, 
\begin{align*}
 \mathcal K := \left\lbrace \bv \in \HT \ : \ \Deform_\gamma \bv = 0 \right\rbrace
\end{align*}
are known as the \emph{Killing fields} of $\gamma$. 
Their $L_2$-orthogonal complement is 
\begin{align*}
 \mathcal K^{\perp}
 :=
 \left\lbrace
    \bg \in L_2(\gamma)^2
    \middle| 
    \ \bg \cdot \bnu = 0 \ a.e. \textrm{ in }\gamma \textrm{ and } \forall \bh \in \mathcal K \ : \int_\gamma  \bg \cdot \bh = 0 
    \right\rbrace .
\end{align*}
Note that the Killing fields are divergence-free 
because the divergence is the trace of the tangential total derivative,
which coincides with the trace of the deformation operator.
Also, because Killing fields correspond to rigid rotations of $\gamma$,
we have ${\rm dim}(\mathcal{K}) \le 3$ with equality holding only if $\gamma$ is a sphere.
Generically for the class of $C^2$ surfaces there holds ${\rm dim}(\mathcal{K}) = 0$.

\subsection{The surface Stokes system}

Given $\bef \in \mathcal K^\perp$, the form of the Stokes system that we consider seeks $\bu \in \HT \cap \mathcal K^\perp$ and $p \in L_{2,\#}(\gamma)$ related by \eqref{weakform}.
We thus define the bilinear form and linear functional 
\begin{align*}
    \forall \bu, \bv \in \HT : 
    \quad 
    a(\bu,\bv)
    &:=
    2 \int_\gamma \Deform_\gamma \bu : \Deform_\gamma \bv
    \\ 
    \forall \bv \in L_2(\gamma)^3: 
    \quad 
    L(\bv)
    &:=
    \int_\gamma \bef \cdot \bv.
\end{align*}
An alternative weak formulation removes the pressure variable 
and incorporates the divergence-free constraint strongly into the velocity space. 
We define 
\begin{align*}
    \VT 
    := 
    \left\lbrace 
        \bv \in \HT 
        \ \middle| \
        \forall q \in L_{2,\#}(\gamma) 
        \ : \ 
        \int_\gamma \textrm{div}_{\gamma} \bv ~ q = 0 
    \right\rbrace
    .
\end{align*}
Thanks to the following Korn-type inequality  \cite{JOR18}
\begin{equation}\label{e:Korn}
    \| \bv \|_{H^1(\gamma)} 
    \lesssim
    \| \Deform_{\gamma} \bv \|_{H^1(\gamma)} + \|P_{\mathcal{K}}\bv\|_{L_2(\gamma)}, 
    \qquad 
    \forall \bv \in {\bf HT}^1,
\end{equation}
where $P_{\mathcal{K}}$ is the $L_2$ projection onto $\mathcal{K}$, the Lax-Milgram theory guarantees that  the velocity $\bu \in \VT \cap \mathcal K^\perp$ is determined uniquely by the relationship
\begin{equation}\label{e:velocity_div_free}
    a(\bu,\bv)
    =
    L(\bv), 
    \qquad 
    \forall \bv \in \VT \cap \mathcal K^\perp. 
\end{equation}
The velocity field $\bu$ is smoother that just in $\HT \subset H^1(\gamma)^3$.
In fact, we recall the elliptic regularity property:  Given $\bg \in \mathcal K^\perp$, the weak solution $\bw \in \VT \cap \mathcal K^{\perp}$ defined by
\begin{equation}
    a(\bw,\bv) = \int_{\gamma} \bg \cdot \bv,
    \qquad 
    \forall \bv \in \VT
\end{equation}
satisfies $\bw \in \HT[2]$ and 
\begin{equation}\label{e:reg_generic}
    \| \bw \|_{\HT[2]} \lesssim \| \bg \|_{L_2(\gamma)}.
\end{equation}
This property holds for any $C^3$ surface $\gamma$ and guarantees that $\bu \in  \HT[2] \subset H^2(\gamma)^3$ with
\begin{equation}\label{e:reg_u}
    \| \bu \|_{H^2(\gamma)} \lesssim \| \bef \|_{L_2(\gamma)}.
\end{equation}
As a simple consequence we get 
\begin{equation}\label{e:killing_reg}
    \mathcal K \subset {\bf HT}^2(\gamma).
\end{equation}
%
%
%
%
%
To recover the pressure, we recall $\bef \in \mathcal K^\perp$ and the inf-sup property \cite{JOR18}
\begin{equation}
    \inf_{q \in L_{2,\#}(\gamma)} 
    \sup_{ \bv \in \HT \cap \mathcal K^{\perp}}
    \frac{\int_\gamma q ~\textrm{div}_\gamma(\bv)}{\| \bv \|_{H^1(\gamma)} \|q\|_{L_2(\gamma)} }
    > 0, 
\end{equation}
which guarantees that the pressure is uniquely determined from the relation
\begin{gather*}
    \forall \bv \in \HT \cap \mathcal K^\perp: 
    \quad
    - 
    \int_\gamma p ~\textrm{div}_\gamma \bv
    = 
    \int_\gamma  \bef \cdot \bv
    -
    \bf 2 \int_\gamma \Deform_\gamma \bu : \Deform_\gamma \bv 
    .
\end{gather*}


\section{A perturbed problem}
\label{sec:todo}  Killing fields are only present on surfaces f symmetry,
so it may be difficult to ascertain the dimension of $\mathcal{K}$ numerically and thus correctly enforce the condition $u \perp \mathcal{K}$.
In order to robustly filter out Killing fields we propose to consider the \emph{perturbed} problem 
of finding $\bu^\varepsilon \in \VT$ such that
\begin{equation}\label{e:valocity_divergence_eps}
    a_\varepsilon(\bu^\varepsilon,\bv) = L(\bv), \qquad \forall \bv \in \VT, 
\end{equation}
where $\varepsilon >0$ and 
\begin{align*}
    a_\varepsilon(\bw,\bv):= a(\bw,\bv) + \varepsilon   \int_\gamma \bw \cdot \bv.
\end{align*}
The coercivity constant associated with the perturbed problem \eqref{e:valocity_divergence_eps} is bounded from below by the coercivity constant of the unperturbed problem. From this, one directly deduces the wellposedness of the perturbed problem and in particular that \eqref{e:valocity_divergence_eps}
has a unique solution.

For $\bv \in \mathcal K$,
we have already observed that $\textrm{tr}(\Deform_\gamma \bv) = \textrm{div}_\gamma \bv = 0$
and because $\bef \in \mathcal K^{\perp}$, 
testing \eqref{e:valocity_divergence_eps} with $\bv \in \mathcal{K}$ readily yields that
\begin{equation}\label{e:ueps_orth}
    \bu^\varepsilon \in \VT \cap \mathcal K^{\perp}.
\end{equation}
This property is critical to derive the following consistency estimate.
\begin{lemma}[Consistency]
Given $f \in \mathcal K^\perp$, let $\bu \in \HT \cap \mathcal K^\perp$ satisfy \eqref{e:velocity_div_free}
and for $\varepsilon >0$, let $\bu^\varepsilon \in \HT$ satisfy \eqref{e:valocity_divergence_eps}.
Then
\begin{align}
    \label{eq:consistency}
    \|\Deform_\gamma (\bu-\bu^\varepsilon)\|_{L_2(\gamma)} +  \|\bu-\bu^\varepsilon \|_{L_2(\gamma)}
    \lesssim   \varepsilon \| \bef \|_{L_2(\gamma)}.
 \end{align}
\end{lemma}

\begin{proof}
 Due to the property \eqref{e:ueps_orth} and the Korn inequality \eqref{e:Korn}, there holds
 \begin{align*}
  \| \bu - \bu^\varepsilon \|_{L_2(\gamma)} 
  \lesssim
  \| \Deform_{\gamma} (\bu - \bu^\varepsilon) \|_{L_2(\gamma)}, 
 \end{align*}
 and so it suffices to prove
 \begin{align*}
  \| \Deform_{\gamma} (\bu - \bu^\varepsilon) \|_{L_2(\gamma)}  \lesssim \varepsilon \| \bef \|_{L_2(\gamma)}.
 \end{align*}
 To see the latter, we subtract the two weak formulations to write
 \begin{align*}
  a(\bu - \bu^\varepsilon, \bv) - \varepsilon \int_\gamma \bu^\varepsilon \cdot \bv = 0.
 \end{align*}
 Using again that $\bu^\varepsilon \in \VT \cap \mathcal K^{\perp}$ and employing the Korn inequality \eqref{e:Korn},
 we deduce that
 \begin{align*}
 \ARD{2} \| \Deform_{\gamma} (\bu - \bu^\varepsilon) \|_{L_2(\gamma)}^2
  &= 
  \varepsilon
  \int_\gamma  \bu^\varepsilon \cdot ( \bu - \bu^\varepsilon ) 
  \\&\leq
  \varepsilon \| \bu^\varepsilon \|_{L_2(\gamma)} \| \bu - \bu^\varepsilon \|_{L_2(\gamma)} 
  \\
  &\lesssim
  \varepsilon \| \Deform_\gamma \bu^\varepsilon \|_{L_2(\gamma)} \| \Deform_\gamma(\bu - \bu^\varepsilon) \|_{L_2(\gamma)}.
 \end{align*}
 The desired result follows from the energy estimate 
 \begin{align*}
  \| \Deform_\gamma \bu^\varepsilon \|_{L_2(\gamma)} \lesssim \| \bef \|_{L_2(\gamma)}.
 \end{align*}
 Here we use that the coercivity constant of $a_{\varepsilon}(\cdot,\cdot)$
 over the space $\VT \cap \mathcal K^{\perp}$
 is bounded from below by the coercivity constant 
 of $a(\cdot,\cdot)$ over that space. 
\end{proof}

We proceed by noting that as for the unperturbed problem, 
the elliptic regularity property \eqref{e:reg_generic} guarantees that $\bu^\varepsilon \in  \HT[2] \subset H^2(\gamma)^3$ with 
\begin{equation}\label{e:reg_u_eps}
\| \bu^\varepsilon\|_{H^2(\gamma)} \lesssim \| \bef \|_{L_2(\gamma)};
\end{equation}
compare with \eqref{e:reg_u}.
It is worth mentioning that the constant hidden in `$\lesssim$' is independent of $\varepsilon$.
Indeed,
observe that for $\bef \in \mathcal K^\perp$, the solution $\bz^\varepsilon \in \VT \cap \mathcal K^\perp$ defined by
\begin{gather*}
    a_\varepsilon (\bz^{\varepsilon},\bv) = \int_\gamma \bef \cdot \bv, \qquad \forall \bv \in \VT,
\end{gather*}
satisfies
\begin{gather*}
    a (\bz^{\varepsilon},\bv) = \int_\gamma \bef \cdot \bv - \varepsilon \int_\gamma \bz^\varepsilon \bv, \qquad \forall \bv \in \VT.
\end{gather*}
Hence, the elliptic regularity property \eqref{e:reg_generic} guarantees that $\bz^\varepsilon \in {\bf HT}^2(\gamma)$ and, along with a Korn inequality and the energy estimate $\| \bz^\varepsilon \|_{H^1(\gamma)} \lesssim \| \bef \|_{L_2(\gamma)}$, we have
\begin{equation}\label{e:elliptic_eps_generic}
    \| \bz^{\varepsilon}\|_{H^2(\gamma)}
    \lesssim \| \bef \|_{L_2(\gamma)} + \varepsilon \| \bz^\varepsilon \|_{H^1(\gamma)} 
    \lesssim \| \bef \|_{L_2(\gamma)}.
\end{equation}
Below we analyze numerical methods using $\varepsilon = h^\alpha$ with various choices of $1 \le \alpha \le 2$
in order to account for Killing fields in the approximation of $\bu$.
Adding an $L_2$ inner product to the bilinear form also helps to achieve a stable interior penalty finite element formulation.

\section{Finite Element Approximations}
\label{sec:fem}
In this section we define a divergence-conforming interior penalty finite element method and prove stability results and basic error estimates for it.

\subsection{Discrete surface}
\label{subsec:discrete_surf}
We assume the existence of a polyhedral surface $\Gamma \subset \mathcal{N}$ embedded in $\mathbb{R}^{3}$ such that all faces of $\Gamma$ are triangular (non-degenerate).  For simplicity we moreover assume that the vertices of $\Gamma$ lie on $\gamma$.
Intuitively, $\Gamma$ is a sufficiently refined interpolation of $\gamma$.  

We denote by $\bMT$ the set of faces of $\Gamma$ and the associated triangulation shape-regularity constant by
\begin{equation}\label{d:shape-reg}
    \sigma_1 := \max_{\barT \in \bMT} \frac{\textrm{diam}(\barT)}{h_{\barT}}, \qquad \textrm{with }h_{\barT} := |\barT|^{1/2}. 
\end{equation}
Further we define the quasi-uniformity constant
\begin{equation}\label{d:quasi-unif}
    \sigma_2 := \max_{\barT\in \bMT} h_{\barT}/ \min_{\barT\in \bMT} h_{\barT}
\end{equation}
and the valence constant
\begin{equation}\label{d:valence}
    M:=\sup_{\bz \in \mathcal V} V(\bz),
\end{equation}
where $\mathcal V$ stands for the set of vertices of $\Gamma$ and $V(\bz)$ the valence of $\bz$. 
The constants appearing in the discussion below might depend on $\sigma_1$, $\sigma_2$, and $M$
but not on the maximal diameter $h:=\max_{\barT\in \bMT} h_{\barT}$.   \ARD{We also let $\MT=\{\bP(\overline{T}):\overline{T} \in \bMT\}$. }  

In addition, we denote by $\bnu_\Gamma : \Gamma \rightarrow \bbR^{3}$ the outward-pointing unit normal to $\Gamma$, 
which is piecewise constant and defined almost everywhere.  
We recall that the normal to $\gamma$ is extended to $\mathcal N$ by the relation $\bnu = \nabla d$
and we assume throughout that $\Gamma$ is transverse to $\gamma$ in the sense that
\begin{align} \label{ass:transverse} \bnu \cdot \bnu_\Gamma \ge c>0, \quad a.e. \quad \textrm{in } \Gamma. \end{align}
Under this geometric non-degeneracy assumption, one can relate the infinitesimal area of the two surfaces \cite{DD07}
: for $v \in L_1(\gamma)$ we have $\int_\Gamma ( v \circ \bP ) \mu = \int_\gamma v$, where for a.e. $\bx \in \Gamma$ 
\begin{align}
    \label{eq:mu_def}
    \mu(\bx):=\bnu(\bx) \cdot \bnu_\Gamma(\bx) \Pi_{i=1}^2 (1-d(\bx) \kappa_i(\bx)),
\end{align}
and $\bP^{-1}$ is the inverse of $\bP:\Gamma \rightarrow \gamma$.  
The definition of $\mathcal{N}$ and the assumption \eqref{ass:transverse} then ensure that $\mu \simeq 1$ on $\Gamma$, which implies the norm equivalency result
\begin{align}
    \label{eq:basic_l2equiv}
    \|u \circ \bP\|_{L_2(\Gamma)} \simeq \|u \|_{L_2(\gamma)}, ~~u \in L_2(\gamma),
\end{align} 
where $\bP:\mathcal N \rightarrow \gamma$ is defined in \eqref{e:P}.


\subsection{Piola transforms}
Below we approximate the solution $\bu$ to the surface Stokes problem via a finite element space 
\begin{gather*}
    {\bf XT}(\MT) \subset H(\divergence;\gamma)
    :=
    \{\bu \in L_2(\gamma)^3 : \bu \cdot \bnu = 0 \ a.e. \textrm{ in }\gamma,\  \divergence_\gamma \bu \in L_2(\gamma)\}.
\end{gather*}
The canonical transformation  for $H(\divergence;.)$ spaces between two surfaces is the Piola transform.
We refer to \cite{steinmann2008boundary,CD16} for its properties in the context of surfaces, which we briefly describe now.  Before doing so, we anticipate that it will be used to map $H(\divergence;\widehat T) \to H(\divergence;\barT) \to H(\divergence;\bP(\barT))$ to define the finite element method on $\gamma$. Here $\widehat T$ is the two dimensional reference simplex and $\barT \in \bMT$.

\ARD{Given sufficiently smooth (open or closed) surfaces $\mathcal{S}_0$ and $\mathcal{S}_1$,} let $\Phi : \mathcal{S}_0 \rightarrow \mathcal{S}_1$ be a diffeomorphism and let $\Phi^{-1}$ be its inverse mapping.  
Let also $D\Phi$ and $D\Phi^{-1}$ be the corresponding tangent maps, that is, 
$D\Phi: \mathbb{T}_0 \rightarrow \mathbb{T}_1$ and $D\Phi^{-1}: \mathbb{T}_1 \rightarrow \mathbb{T}_0$, 
where $\mathbb{T}_i$ is the tangent space of $\mathcal{S}_i$.  
Finally, let $\mu$ formally satisfy $\mu d \sigma_0 = d \sigma_1$, where $d \sigma_i$ is surface measure on $\mathcal{S}_i$. 
For $\bq_0 \in H(\divergence;\mathcal S_0)$, the surface Piola transform  $\mathcal P_\Phi \bq_0$ is given by
\begin{gather*}
    \piola_{\Phi} \bq_0:= \mu^{-1} D\Phi (\bq_0), ~\quad  \bq_0 \in H(\divergence; \mathcal{S}_0).
\end{gather*}
Note that $\mathcal P_{\Phi^{-1}} \bq_1$ for $\bq_1 \in  H(\divergence; \mathcal{S}_1)$ is defined similarly and satisfies
\begin{gather*}
    \piola_{\Phi^{-1}} \bq_1 = \mu D\Phi^{-1} (\bq_1), ~\quad  \bq_1 \in H(\divergence; \mathcal{S}_1).
\end{gather*}
The identity 
\begin{alignat}{1}
    \label{piola_div}
    \divergence_{\mathcal{S}_0} \bq_0  = \mu\, \divergence_{\mathcal{S}_1} \bq_1,
\end{alignat}
is valid for $\bq_0 \in H(\divergence,S_0)$ and $\ARD{\mathcal{P}_\Phi \bq_0=}\bq_1  \in H(\divergence,S_1)$ \cite{steinmann2008boundary}. 
Thus $\mathcal P_\Phi: H(\divergence;\mathcal S_0) \rightarrow H(\divergence;\mathcal S_1)$ 
and $\mathcal P_{\Phi^{-1}}: H(\divergence;\mathcal S_1) \rightarrow H(\divergence;\mathcal S_0)$ 
are bounded mappings.

Let us now specialize these relationships to the case $\mathcal S_0= \barT \in \bMT$, $\mathcal S_1= T:= \bP(\barT) \subset \gamma$ and $\Phi=\bP$.
In this context, $D\bP(\bq_0) = \nabla \bP \bq_0 = (\Pi- d \bH)\bq_0$, $\mu$ is given by \eqref{eq:mu_def} 
and  $\Phi^{-1}$ is the inverse of $\bP$ viewed as a mapping $\Gamma \rightarrow \gamma$.
One can check that $D\Phi^{-1}(\bq_1) = [\bI - \frac{\bnu \otimes \bnu_\Gamma}{\bnu \cdot \bnu_\Gamma} ] [\bI - d \bH]^{-1}\bq_1$
(see for instance Lemma~20 in \cite{bonito2019finite}).  
Hence, given  $\barq \in H(\divergence; \Gamma)$ tangential to $\Gamma$, 
we define $\bq \in H(\divergence; \gamma)$ tangential to $\gamma$ by $\bq := \piola_\bP \barq$ 
where for $\barx \in \Gamma$ and $\bx = \bP(\barx)$
\begin{align}
    \label{eq10}
    \bq(\bx) = \mu(\barx)^{-1} [\Pi(\barx)-d(\barx)\bH(\barx)] \barq(\barx).
\end{align}
Similarly, we define $\barq  := \piola_{\bP^{-1}} \bq$ by
\begin{align}
    \label{eq9}
    \barq (\barx) 
    = 
    \mu(\barx)
    \left [\bI - \frac{\bnu(\barx) \otimes \bnu_\Gamma(\barx)}{\bnu(\barx) \cdot \bnu_\Gamma(\barx)} \right ]
    [\bI - d(\barx) \bH(\barx)]^{-1} \bq (\bx).
\end{align}


The next lemma relates norms of vector fields and their Piola transforms between $\Gamma$ and $\gamma$.  

\begin{lemma}
    For $\barT \in \bMT$, set  $T:=\bP(\barT)$.
    Assume that  $\bq \in \HTloc{T}$ and $\barq \in \HTloc{\barT}$ are related by \eqref{eq10} and \eqref{eq9}.  Then
    \begin{align}
        \label{L2_equiv}
        \|\bq\|_{L_2(T)} \simeq \|\barq \|_{L_2(\barT)}, ~~~~\|\bq\|_{H^1(T)} \simeq \|\barq\|_{H^1(\barT)}.
    \end{align}
    If in addition each component of $\barq$ is affine, then
    \begin{align}
        \label{H2_bound}
        \|D_\gamma \nabla_\gamma \bq\|_{L_2(T)} \lesssim \|\bq\|_{H^1(T)},
    \end{align}
    where by $D_\gamma$  we denote the scalar tangential gradient on $\gamma$ acting componentwise.
\end{lemma}
\begin{proof}
      The first relationship in \eqref{L2_equiv} follows upon noting that the matrices multiplying $\barq$ and $\bq$ in \eqref{eq10} and \eqref{eq9} are bounded due to the assumptions in Section~\ref{sec:surf_rep} and then applying the norm equivalency result \eqref{eq:basic_l2equiv}.
    In particular, the definition of $\mathcal N$ and the assumption $\bnu \cdot \bnu_\Gamma \ge c >0$ 
    imply that  the eigenvalues of $\bI-d\bH$ lie in the interval $(1/2, 3/2)$ for $\bx \in \mathcal{N}$  and that $\mu \simeq 1$.

    In order to prove the $H^1$ estimate, we define
    \begin{align*}
        \bM:=\mu \left [ \bI-\frac{\bnu\otimes \bnu_\Gamma}{\bnu \cdot \bnu_\Gamma} \right] [\bI-d\bH]^{-1}.
    \end{align*}
    Suppressing the dependence on $\barx \in \barT \subset \Gamma$, we  have using the product and chain rules that
    \begin{gather*}
        |\nabla_\Gamma \barq|
        =
        |\Pi_\Gamma \nabla \barq \Pi_\Gamma| 
        \le
        |\nabla \barq|
        =
        |\nabla(\bM \bq \circ \bP)| 
        \le
        |\nabla \bM||\bq| + |\bM \nabla \bq \nabla \bP|,
    \end{gather*}
    where $\Pi_\Gamma:= \bI - \bnu_\Gamma \otimes \bnu_\Gamma$.
    Here and in what follows we implicitly use the canonical extension of $\barq$ to a neighborhood of $\Gamma$, 
    that is, we extend each component of $\barq$ so that it is constant in the direction of the normal $\bnu$ to $\gamma$.  

    The geometric relations provided in Section~\ref{sec:surf_rep} and elementary calculations yield
    $| \nabla \bM|\lesssim \|d\|_{C^3(\mathcal{N})}$ and $| \nabla \bP| = |\Pi-d\bH|= |\Pi(\bI-d\bH)| \lesssim  \|d\|_{C^2(\mathcal{N})}$.  
    Furthermore, because $\bH \bnu = 0$, we have $[\bI-d\bH]^{-1} \bnu=\bnu$, and we also easily compute that $\left [ \bI-\frac{\bnu\otimes \bnu_\Gamma}{\bnu \cdot \bnu_\Gamma} \right] \Pi = \left [ \bI-\frac{\bnu\otimes \bnu_\Gamma}{\bnu \cdot \bnu_\Gamma} \right]$.    Thus $\bM=\bM\Pi $, and
    \begin{gather*}
        |\nabla_\Gamma \barq| \lesssim |\bq|+|\bM \Pi \nabla \bq \Pi [\bI-d\bH]| \lesssim |\bq|+|\nabla_\gamma \bq|.
    \end{gather*}
    Combining this inequality with the already established equivalence of $L_2$ norms yields that for $\barT \in \bMT$,
    \begin{gather*}
        \|\barq\|_{H^1(\barT)}  \lesssim \|\bq\|_{H^1(T)}.
    \end{gather*}

    For the converse inequality, define $\bL:=\frac{1}{\mu} [\Pi-d\bH]$ so that $\bq \circ \bP =\bL \barq$.  First note that because $\barq$ is assumed to be tangential to $\Gamma$, $\barq =\Pi_\Gamma \barq$.  Moreover, $\Pi_\Gamma= \bI - \bnu_\Gamma \otimes \bnu_\Gamma$ is constant elementwise.  Thus on $\Gamma$ we have $\nabla \barq=\nabla(\Pi_\Gamma \barq)=\Pi_\Gamma \nabla \barq$.   Let $\bP^{-1}:\gamma \rightarrow \Gamma$ denote the inverse mapping of $\bP:\Gamma \rightarrow \gamma$.  By (2.18) of \cite{DD07} and the definition of tangential scalar gradient we have that for $\overline{U} \in H^1(\Gamma)$, $\nabla (\overline{U} \circ \bP^{-1})=\nabla_\Gamma \overline{U}   \Pi_\Gamma [\bI-\frac{\bnu \otimes \bnu_\Gamma}{\bnu_\Gamma \cdot \bnu}]$.  When combined with the identity $\nabla \barq=\Pi_\Gamma \nabla \barq$, this yields that 
    \begin{align}
        \label{grad_ident}
        \begin{aligned}
            \nabla (\barq \circ \bP^{-1}) &= (\nabla \barq) \circ \bP^{-1} \Pi_\Gamma [\bI-\frac{\bnu \otimes \bnu_\Gamma}{\bnu_\Gamma \cdot \bnu}] = \Pi_\Gamma \nabla \barq \Pi_\Gamma[\bI-\frac{\bnu \otimes \bnu_\Gamma}{\bnu_\Gamma \cdot \bnu}]
            \\ &  = ( \nabla_\Gamma \barq) \circ \bP^{-1} [\bI-\frac{\bnu \otimes \bnu_\Gamma}{\bnu_\Gamma \cdot \bnu}] .
        \end{aligned}
    \end{align}
    Thus we have that
    \begin{align}
        \label{interm_bound}
        \begin{aligned}
            |\nabla_\gamma \bq|& =|\Pi \nabla \bq \Pi| \le |\nabla \bq|  = |\nabla [ \bL \barq \circ \bP^{-1}]|
            \\ & \le | \nabla \bL||\barq\circ \bP^{-1} |  + |\bL \nabla (\barq \circ \bP^{-1})|
            \\ & = | \nabla \bL||\barq\circ \bP^{-1} | + |\bL \Pi_\Gamma (\nabla \barq)\circ\bP^{-1}  \Pi_\Gamma [\bI-\frac{\bnu \otimes \bnu_\Gamma}{\bnu_\Gamma \cdot \bnu}]| 
            \\ & \lesssim |\barq \circ \bP^{-1}| + |(\nabla_\Gamma \barq) \circ \bP^{-1}|.  
        \end{aligned}
    \end{align}
    Using equivalence of norms then yields
    \begin{gather*}
     \|\bq\|_{H^1(T)} \lesssim \|\barq\|_{H^1(\barT)}.
    \end{gather*}
    We finally prove \eqref{H2_bound}.  Using the calculation in \eqref{interm_bound} and the boundedness of the derivatives of $\Pi$ (since $\gamma$ is $C^4$), we have that
    \begin{gather*}
     |D_\gamma \nabla_\gamma \bq|=|D_\gamma (\Pi \nabla \bq \Pi)| \lesssim |\nabla \bq| + |D_\gamma \nabla \bq| \lesssim |\barq|+ |\nabla_\Gamma \barq| + |D_\gamma \nabla \bq|.
    \end{gather*}
    Recalling that $\bq=\bL \barq \circ \bP^{-1}$, we have that 
    \begin{gather*}
     |D_\gamma \nabla \bq| \lesssim |D^2 \bL||\barq \circ \bP^{-1}| + |\nabla \bL| |\nabla (\barq\circ \bP^{-1})| + |\bL| |D_\gamma \nabla  (\barq \circ \bP^{-1})|.
    \end{gather*}
    Because $\gamma$ and therefore also $d$ is $C^4$, 
    elementary calculations yield $|\nabla^j \bL| \lesssim 1$ ($j = 0,1,2$) 
    and $|\nabla[\bI-\frac{\bnu \otimes \bnu_\Gamma}{\bnu_\Gamma \cdot \bnu}]| \lesssim 1$. 
    Employing \eqref{grad_ident} and the product rule thus yields
    \begin{align}
        \begin{aligned} 
            |D_\gamma \nabla (\barq \circ \bP^{-1})| 
            &=
            \Big|D_\gamma \Big ( (\nabla_\Gamma \barq) \circ \bP^{-1}  [\bI-\frac{\bnu \otimes \bnu_\Gamma}{\bnu\cdot \bnu_\Gamma}] \Big) \Big| 
            \\&\lesssim
            |\nabla_\Gamma \barq|  + |D_\gamma ( (\nabla_\Gamma \barq) \circ \bP^{-1}) |.  
        \end{aligned}
    \end{align}
    Recalling that $\barq$ is componentwise affine and $\Pi_\Gamma$ is constant, 
    we have that $\nabla_\Gamma \barq$ is constant.  
    Thus  $|D_\gamma \big ( (\nabla_\Gamma \barq) \circ \bP^{-1}\big ) | = 0$.  
    Collecting the previous three inequalities then yields
    \begin{gather*}
     |D_\gamma \nabla_\gamma \bq| \lesssim |\barq \circ \bP^{-1}| + |(\nabla_\Gamma \barq)\circ \bP^{-1}|,
    \end{gather*}
    which yields the desired result after employing equivalence of $L_2$ norms and then the $H^1$ equivalence in \eqref{L2_equiv}.  
\end{proof}

\subsection{Finite element space}
The finite element method we advocate is based on a Brezzi-Douglas-Marini finite element triplet. We denote by ${\bf XT}(\MT)$ the resulting finite element space, which we describe in detail below.  The essential properties of our method are that it is divergence conforming and consists of tangential vector fields, that is, $\bV \cdot \bnu = 0$ for all $\bV \in {\bf XT}(\mathcal{T})$.  However, members of ${\bf XT}(\mathcal{T})$ are only continuous in the normal direction across element boundaries, so ${\bf XT}(\mathcal{T}) \not \subset H^1(\gamma)^3$ and penalization is necessary in order to achieve stability.  

To date no conforming finite element subspaces of $\HT$ have been constructed, and it is not clear how this construction might be approached. In the context of standard finite element constructions, the difficulty lies in enforcing strict continuity across element boundaries in order to ensure $H^1$ continuity while simultaneously ensuring that finite element functions are tangent vectors.  \ARD{A typical strategy for constructing finite element spaces on $\gamma$ is to define corresponding spaces on the discrete surface $\Gamma$ and transform them to $\gamma$ via a canonical transformation.  However, it is not clear that the underlying space ${\bf HT}^1(\Gamma)$ on the discrete surface is well-defined.}  In particular, a natural characterization is ${\bf HT}^1(\Gamma)=\Pi_\Gamma [H^1(\Gamma)]^3$.  Because the vector field $\Pi_\Gamma={\bf I}-\bnu_\Gamma \otimes \bnu_\Gamma$ only possesses $L_\infty$ regularity for the Lipschitz surface $\Gamma$, Sobolev multiplier properties lead us to expect that $\Pi_\Gamma [H^1(\Gamma)]^3 \not\subset [H^1(\Gamma)]^3$.  For example, the product of a constant function (which lies in $H^1$) by a piecewise constant function lying in $L_\infty$ is not in $H^1$.  Thus it is not even immediately clear that there is a meaningful definition of the underlying space ${\bf HT}^1(\Gamma)$.


Let $\widehat{T}\subset \mathbb{R}^2$ be the standard reference triangle. 
We begin by letting ${\bf XT}(\widehat{T})=[\mathbb{P}_1]^2$ be the affine vector functions on $\widehat{T}$.  
Given $\barT \in \bMT$, let $A_{\barT} :\widehat{T} \rightarrow \barT$ be the natural affine reference transformation. 
We then define ${\bf XT}(\bMT)$  to be the set of vector functions $\barq:\Gamma \rightarrow \mathbb R^3$ 
such that for all $\barT \in \bMT$ we have $\barq|_{\barT}=\piola_{A_{\barT}} \widehat{\bq}$ for some $\widehat{\bq} \in {\bf XT}(\widehat{T})$, 
and in addition the members of ${\bf XT}(\bMT)$ possess normal continuity.  
That is, given $\barT,\barT' \in \bMT$ sharing an edge $\bare$, $\barq_{\barT} \cdot \barbn+\barq_{\barT'} \cdot \barbn' = 0$, 
where $\barbn$ and $\barbn'$ are the outward-pointing unit conormals calculated on $\bare$ from $\barT$ and $\barT'$, respectively.  
Note that in contrast to the case of Euclidean domains there holds in general that $\barbn \neq -\barbn'$.  
Let also $\V(\bMT)=\{\barv \in L_2(\Gamma):\barv|_{\barT} \in \mathbb{P}_0(\barT), \barT \in \bMT\}$ be the set of discontinuous piecewise constants.  
Finally, let 
\begin{align*}
{\bf XT}(\T)=\piola_{\bP} {\bf XT}(\bMT),
\end{align*}
and recalling the definition \eqref{eq:mu_def} we let $\V(\T)=\{\frac{1}{\mu} \barv \circ \bP^{-1}:\barv \in \V(\bMT)\}$. 
Note that while the space $\V(\bMT)$ defined on the discrete surface consists of piecewise constants, 
the space $\V(\T)$ defined on the continuous surface $\gamma$ does not due to the presence of the term $\frac{1}{\mu}$ in its definition. 
This definition ensures that a commuting diagram property holds, 
and in particular that $\divergence_\gamma {\bf XT}(\T)=\V(\T)$; cf. \eqref{piola_div}.   
Finally, we denote by $\VT[\T]$ its subspace of tangential divergence free vectors
\begin{gather*}
    \VT[\T]:=\left\lbrace  \bq \in {\bf XT}(\T) \ : \ \divergence_\gamma \bq = 0 \quad \textrm{a.e. in }\gamma   \right\rbrace.
\end{gather*}

Below we also recall the standard BDM interpolant; we defer this definition and statement of approximation properties of the interpolant until we have defined an appropriate discrete norm related to our interior penalty method.


\subsection{Interior penalty method}

We propose to consider an interior penalty type method 
to construct the discrete approximation $\bU^\varepsilon \in \VT[\MT]$ of $\bu^\varepsilon$ given by \eqref{e:valocity_divergence_eps}.
This reads: given $\rho >0$,  find $\bU^\varepsilon \in \VT[\MT]$ such that
\begin{align}
    \label{FE_def}
    a_\varepsilon(\bU^\varepsilon,\bV)
    +
    j(\bU^\varepsilon,\bV)
    = 
    \int_\gamma \bef \cdot \bV
\end{align}
for all $\bV \in  \VT[\MT]$. Here
\begin{gather*}
    j(\bW,\bV)
    := 2 \Big [
    -\int_{\Sigma } 
    \{ \Deform_\gamma \bW \cdot \bn \} \cdot \lbrack \bV \rbrack
    - 
    \int_\Sigma 
    \{ \Deform_\gamma \bV \cdot \bn \} \cdot \lbrack \bW \rbrack
    + 
    \frac{\rho}{h} 
    \int_{\Sigma}
    \lbrack \bW \rbrack \cdot  \lbrack \bV \rbrack \Big ] ,
\end{gather*}
where we denote by $\Sigma$ the set of all edges of $\MT$,
$h := \max_{\barT \in \bMT} \textrm{diam}(\barT)$, 
$\bn|_{\partial T}$ is the outside pointing conormal of $T$, 
and 
\begin{align*}
 \{ \Deform_\gamma \bW \cdot \bn \} 
 :=
 \frac 1 2 
 \left( 
    \Deform_\gamma \bW^+ \cdot \bn^+ - \Deform_\gamma \bW^- \cdot \bn^-
 \right),
 \qquad
 \lbrack \bW \rbrack := V^+-V^-.
\end{align*}
In addition, we have used the convention 
\begin{gather*}
   a_\varepsilon(\bU^\varepsilon,\bV) = 2 \sum_{T \in \MT} \int_T \Deform_T \bU^\varepsilon : \Deform_T \bV+\varepsilon \int_\gamma \bU^\varepsilon \cdot \bV.
\end{gather*}
We denote by $\vvvert . \vvvert_{1,h}$ the discrete norm defined by
\begin{align*}
 \vvvert \bV \vvvert_{1,h}^2
 :=
 \sum_{T \in \MT} \int_T |\Deform_T \bV|^2
 + 
 \frac{\rho}{h} \| \lbrack \bV \rbrack\|_{L_2(\Sigma)}^2
\end{align*}
and also define the weighted discrete energy norm
\begin{gather*}
     \vvvert \cdot \vvvert_{1,h,\varepsilon}^2=\vvvert \cdot \vvvert_{1,h}^2 + \varepsilon \|\cdot \|_{L_2(\gamma)}^2.
\end{gather*}
When solving \eqref{FE_def} computationally, we employ the full perturbed Stokes system: 
Find $(\bU^\varepsilon, P) \in {\bf XT}(\T) \times \V_\# (\T)$ such that
\begin{align*} 
    \begin{aligned}
        a_\varepsilon(\bU^\varepsilon,\bV)
        +
        j(\bU^\varepsilon,\bV)-\int_\gamma P \, \divergence_\gamma \bV& =\int_\gamma {\bf f} \cdot \bV, ~~~\forall \bV \in {\bf XT}(\T),
        \\ \int_\gamma q \, \divergence_\gamma \bU^\varepsilon &= 0, ~~~\forall q \in \V_\#(\T).
    \end{aligned}
\end{align*}
However, the divergence-conforming nature of our finite element space allows us to consider only the reduced system \eqref{FE_def} for purposes of establishing error estimates.  Approximation of the pressure is not a focus of this work, but error estimates could be established by proving an appropriate inf-sup inequality.

\subsection{The BDM interpolant and a Discrete Korn-type inequality}

Let $\widehat{I}: [H^1(\widehat{T})]^2 \rightarrow {\bf XT}(\widehat{T})$ be the standard interpolant for the BDM space on the reference element $\widehat{T}$. The degrees of freedom of $\widehat{I}$ consist of moments of normal components on element edges \cite[p. 125]{BF91}:
\begin{align}
    \label{bdm_dof}
    \int_{\partial \widehat{T}} (\widehat{\bq} -\widehat{I} \widehat{\bq}) \cdot \widehat{\eta} \,p_1  = 0, ~~p_1 \in R_1(\partial \widehat{T}),
\end{align}
where $R_1$ consists of functions which are affine on each edge of $\widehat{T}$ but not necessarily continuous.    Let then $I_\T:\HT \rightarrow {\bf XT}(\T)$ be given elementwise by $I_\T \bq= \piola_{\bP \circ A_T} \widehat{I} (\piola_{A_T^{-1} \circ \bP^{-1}} \bq)$.    There holds the commuting diagram property
\begin{align}
    \label{commute}
    \divergence_\gamma I_\T \bq = \pi_\V \divergence_\gamma \bq,
\end{align}
where $\pi_\V$ is the $L_2$ projection onto $\V$. 

\begin{lemma}
    \label{lem:approx_properties}
    Given $T \in \T$ and $\bq \in {\bf HT}^1(T) \cap [H^2(T)]^{3}$, there holds
    \begin{align}
        \label{eq:approx}
        \begin{aligned}
        \|\bq-I_\T \bq\|_{L_2(T)} & + h\|\nabla_\gamma(\bq-I_\T \bq)\|_{L_2(T)}+ h\vvvert\bq-I_\T \bq\vvvert_{1,h} \ARD{+h^2|I_\T \bq|_{H^2(T)} }  
      \\ &   \lesssim h^2 \|\bq\|_{H^2(T)}.
        \end{aligned}
    \end{align}
\end{lemma}
\begin{proof}
    We briefly outline the proof, which is mostly standard.  
    For $\barT \in \overline{\T}$ and $\barq \in  {\bf HT}^1(\barT) \cap [H^2(\barT)]^{3}$, the estimate 
    \begin{gather*}
        \|\barq-I_{\barT} \barq\|_{L_2(\barT)} + h\|\nabla_{\barT}(\barq-I_{\barT} \barq)\|_{L_2(\barT)} \lesssim h^2 \|\barq\|_{H^2(\barT)}
    \end{gather*}
    is essentially contained in \cite[Proposition 3.6]{BF91}.
    Transforming these inequalities to $T \in \T$ via $\piola_{\bP}$ yields the desired bounds for the $L_2$ and $H^1$ (derivative) terms in \eqref{eq:approx} on $T \in \T$.    Applying standard arguments involving a scaled trace inequality along with the previously proved bounds for the volume norms in \eqref{eq:approx} yields
    \begin{equation}\label{e:handle_jump}
         h \| \lbrack \bq-I_\T \bq \rbrack\|_{L_2(\Sigma)}^2 \lesssim h^4 \| \bq \|^2_{H^2(T)}.  
    \end{equation}
        %
\ARD{The bound for $h^2 |I_\T \bq|_{H^2(T)}$ follows from \eqref{H2_bound} and the already proved $H^1$ bounds.}   \end{proof}

In order to analyze our method we shall need the following discrete Korn-type inequality relating the discrete $H^1$ norm
\begin{gather*}
     \| \bq \|_{H^1_h(\gamma)}:= \left(\sum_{T \in \MT} \int_T | \nabla_\gamma \bq |^2 \right)^{1/2} + \| \bq \|_{L_2(\gamma)}
\end{gather*}
and $\vvvert \bq \vvvert_{1,h}+\|\bq\|_{L_2(\gamma)}$ for functions lying in ${\bf XT}[\T]$.   The proof is deferred to Appendix~\ref{sec:Korn_app}.

\begin{lemma}
\label{lem:disc_Korn}
    We have for all $\bq \in {\bf XT}(\T)$ that
    \begin{align}
        \label{eq:disc_Korn}
        \begin{aligned}
        \|\bq\|_{H_h^1(\gamma)} & \lesssim \|{\rm Def}_{\gamma,h} \bq\|_{L_2(\gamma)} + h^{-1/2}\|[\bq]\|_{L_2(\Sigma)}+ \|\bq\|_{L_2(\gamma)} 
        \\ & \lesssim \vvvert\bq\vvvert_{1,h} + \|\bq\|_{L_2(\gamma)}.
    \end{aligned}
    \end{align}
\end{lemma}

\subsection{Stability and energy error bounds for the IP method}
\label{subsec:basic_estimates}

We first prove stability of the interior penalty method.  
\begin{lemma}
    \label{lem:well-posed}
    Provided $\rho$ is sufficiently large and $\epsilon \ge h^2$, 
    we have for $\bV, \bW \in {\bf XT}(\T)$
    \begin{align}
    \label{coercivity}
    a_\epsilon(\bW,\bW) + j(\bW, \bW) \gtrsim \vvvert\bW\vvvert_{1,h,\varepsilon}^2 \end{align}
    and 
    \begin{equation}\label{e:cont_dg}
        a_{\varepsilon}(\bW,\bV)
        +
        j(\bW,\bV)
        \lesssim
        \vvvert \bW \vvvert_{1,h,\varepsilon}
        \vvvert \bV \vvvert_{1,h, \varepsilon}.
    \end{equation}
    Consequently 
    \begin{align}
        \label{stability}
        \vvvert \bU^\varepsilon \vvvert_{1,h} +\varepsilon^{1/2} \|\bU^\varepsilon\|_{L_2(\gamma)} \lesssim  \| \bef \|_{L_2(\gamma)}.
    \end{align}
\end{lemma}

\begin{proof}
  We first calculate that for $\delta>0$,
 \begin{align}
 \label{j_calc}
 \begin{aligned}
\frac{1}{2}  j(\bW,\bW)
  &
  = 
  -2
  \int_{\Sigma } 
  \{ \Deform_\gamma \bW \cdot \bn \} \cdot \lbrack \bW \rbrack
  +
  \frac{\rho}{h} 
  \int_{\Sigma}
  \lbrack \bW \rbrack^{2} 
  \\&
\ge
  -
 \sum_{T \in \T} \int_{ \partial T } 
  h^{1/2}
  | \Deform_\gamma \bW | h^{-1/2} |\lbrack \bW \rbrack|
  +
  \frac{\rho}{h} 
  \int_{\Sigma}
  \lbrack \bW \rbrack^{2} 
  \\&
\ge
  - 
 \frac{1}2 h \delta \sum_{T \in \T} \int_{\partial T} 
| \Deform_\gamma \bW |^{2} 
  +
\left ( \frac{\rho}{h} -\frac{1}{2\delta h} \right)  
  \int_{\Sigma}
  \lbrack \bW \rbrack^{2}.
  \end{aligned}
 \end{align}
 Using a scaled trace inequality, \eqref{H2_bound}, and the discrete Korn inequality \eqref{eq:disc_Korn} while recalling the definition of $\vvvert \cdot \vvvert_{1,h}$ and that $\varepsilon \ge h^2$ yields that for $T \in \T$
\begin{align}
\label{eq130}
\begin{aligned}
& \sum_{T \in \T}   h \int_{\partial T}   |\Deform_\gamma \bW|^2  \lesssim \sum_{T\in \T} \left(\|\Deform_\gamma \bW\|^2_{L_2(T)} + h^2 \|D_\gamma \Deform_\gamma \bW\|_{L_2(T)}^2\right)
 \\ & \lesssim \sum_{T \in \T}\|\Deform_\gamma \bW\|_{L_2(T)}^2 + h^2 \|\bW\|_{H^1_h(\gamma)}^2 
 \\ & \lesssim  \sum_{T \in \T} \|\Deform_\gamma \bW\|_{L_2(T)}^2 + h^2 (\| {\rm Def}_\gamma \bW\|_{L_2(\gamma)}^2+ h^{-1}  \|[\bW]\|_{L_2(\Sigma)}^2+ \|\bW\|_{L_2(\gamma)}^2) 
 \\& \lesssim  \sum_{T \in \T}\|\Deform_\gamma \bW\|_{L_2(T)}^2 + h \int_\Sigma \lbrack \bW \rbrack^2 + \varepsilon \|\bW\|_{L_2(\gamma)}^2.
 \end{aligned}
\end{align} 

 Thus for some $C>0$ and any $\delta>0$,
 \begin{align*}
 \begin{aligned}
\frac{1}{2} j(\bW, \bW) & \ge -C\delta (\|\Deform_{\gamma,h} \bW\|_{L_2(\gamma)}^2+ \varepsilon \|\bW\|_{L_2(\gamma)}^2) 
\\ & + \left ( \frac{\rho}{h} - \frac{1}{2\delta h}- C \delta h \right ) \| \lbrack \bW \rbrack \|_{L_2(\Sigma)}^2.
\end{aligned}
 \end{align*}
Taking $C\delta$ sufficiently small and subsequently $\rho$ sufficiently large and adding this inequality to $a_\epsilon (\bW, \bW)$ then yields the coercivity estimate \eqref{coercivity}.   The  continuity estimate \eqref{e:cont_dg} follows similarly, and the stability estimate \eqref{stability} then follows by standard arguments.  
\end{proof}


\ARD{We now prove energy error estimates that are optimal so long as $\varepsilon \lesssim h$.  Below we use a subscript ``$h$'' to denote an elementwise differential operator, e.g., ${\rm Def}_{\gamma, h}$.  }


\begin{lemma}
Assume that $\gamma$ is of class $C^4$, $\bef \in \mathcal K^\perp$, and \ARD{$h^2 \leq \varepsilon \leq 1$}. 
Then the solutions $\bu$ resp. $\bu^\varepsilon$ defined by \eqref{e:velocity_div_free} resp. \eqref{e:valocity_divergence_eps} satisfy 
 \begin{gather}
  \label{e:FEM_approx} 
  \vvvert\bu^\varepsilon - \bU^\varepsilon \vvvert_{1,h,\varepsilon} \ARD{+h \|D_{\gamma, h} {\rm Def}_{\gamma, h} (\bu^\varepsilon - \bU^\varepsilon)\|_{L_2(\gamma)} }
  \lesssim  h  \| \bef \|_{L_2(\gamma)} .
  \\
  \label{eq:final_approx}
  \vvvert\bu - \bU^\varepsilon \vvvert_{1,h,\varepsilon}
  \lesssim
( h+\varepsilon) \| \bef\|_{L_2(\gamma)}.
 \end{gather}
\end{lemma}

\begin{proof}  \ARD{Let $\bv, \bw \in  H({\rm div}; \gamma)$ be elementwise in ${\bf HT}^1(\gamma) \cap [H^2(\gamma)]^3$, as for ${\bf u}^\varepsilon$ and any $\bV \in {\bf XT}(\T)$.  Computing as in \eqref{j_calc}, the  first line of \eqref{eq130} yields for such $\bv, \bw$ that
\begin{align*}
\begin{aligned}
j(\bv, \bw) 
& \lesssim \left ( \|{\rm Def}_{\gamma,h} \bv\|_{L_2(\gamma)} + h \|D_{\gamma,h} {\rm Def}_{\gamma, h} \bv\|_{L_2(\gamma)} + h^{-1/2} \rho^{1/2} \|[ \bv] \|_{L_2(\gamma)}  \right )
\\ & ~~~~   \cdot \left ( \|{\rm Def}_{\gamma,h} \bw\|_{L_2(\gamma)} + h \|D_{\gamma,h} {\rm Def}_{\gamma, h} \bw\|_{L_2(\gamma)} + h^{-1/2} \rho^{1/2} \|[ \bw] \|_{L_2(\gamma)}  \right ) .
\end{aligned}
\end{align*}
Thus 
\begin{align}
\label{dg_continuity} 
\begin{aligned}
a_\varepsilon(\bv, \bw) + j(\bv, \bw) & \lesssim  \left ( \vvvert \bv \vvvert_{1,h,\varepsilon} + h \|D_{\gamma, h} {\rm Def}_{\gamma, h} \bv\|_{L_2(\gamma)} \right ) 
\\&  ~~~~\cdot \left (  \vvvert \bw \vvvert_{1,h,\varepsilon} + h \|D_{\gamma, h} {\rm Def}_{\gamma, h} \bw\|_{L_2(\gamma)} \right ).
\end{aligned}
\end{align}
We next note that 
\begin{equation}\label{e:GALO}
a_\varepsilon(\bu^\varepsilon -\bU^\varepsilon, \bV) + j(\bu^\varepsilon-\bU^\varepsilon, \bV)=0, \qquad \forall \bV \in {\bf VT}(\T),
\end{equation} 
and that $\bU^\varepsilon, I_\T \bU^\varepsilon \in {\bf VT}(\T)$ since by the commuting diagram property \eqref{commute}, ${\rm div}_\gamma I_\T \bu^\varepsilon = \pi_{\mathbb{V}} {\rm div}_\gamma \bu^\varepsilon=0$.  Applying the coercivity estimate \eqref{coercivity}, \eqref{e:GALO} with $\bV =  I_\T \bu^\varepsilon - \bU^\varepsilon$, \eqref{dg_continuity}, and \eqref{H2_bound} while recalling that $h \lesssim \sqrt{\varepsilon}$ yields
\begin{align}
\begin{aligned}
& \vvvert  I_\T \bu^\varepsilon- \bU^\varepsilon \vvvert_{1,h,\varepsilon}^2   \lesssim a_\varepsilon(I_\T \bu^\varepsilon -\bu^\varepsilon, I_\T \bu^\varepsilon -\bU^\varepsilon) +  j(I_\T \bu^\varepsilon -\bu^\varepsilon, I_\T \bu^\varepsilon -\bU^\varepsilon) 
\\ & \lesssim \left ( \vvvert I_\T \bu^\varepsilon -\bu^\varepsilon \vvvert_{1,h, \varepsilon} + h\|D_{\gamma, h} {\rm Def}_{\gamma, h} (I_\T \bu^\varepsilon -\bu^\varepsilon) \|_{L_2(\gamma)}\right ) \vvvert I_\T \bu^\varepsilon - \bU^\varepsilon \vvvert_{1,h,\varepsilon}.
\end{aligned}
\end{align} 
The desired estimate for the first term in \eqref{e:FEM_approx} follows upon dividing through by $\vvvert  I_\T \bu^\varepsilon- \bU^\varepsilon \vvvert_{1,h,\varepsilon}$, applying the approximation estimate \eqref{eq:approx}, and using $H^2$ regularity \eqref{e:reg_u_eps}.  To estimate the second term in \eqref{e:FEM_approx}, we insert $\pm I_\T \bu^\varepsilon$, apply the triangle inequality, and apply \eqref{eq:approx} and \eqref{e:reg_u_eps} to find 
\begin{align*}
h \|D_{\gamma, h} {\rm Def}_{\gamma, h} (\bu^\varepsilon - \bU^\varepsilon)\|_{L_2(\gamma)} \lesssim h \|{\bf f}\|_{L_2(\gamma)} + h \|D_{\gamma, h} {\rm Def}_{\gamma, h} (I_\T \bu^\varepsilon - \bU^\varepsilon)\|_{L_2(\gamma)}. 
\end{align*}
To complete the proof of \eqref{e:FEM_approx} we employ \eqref{H2_bound}, the discrete Korn inequality \eqref{eq:disc_Korn}, the already established bound for the first term in \eqref{e:FEM_approx}, approximation bounds, and $\varepsilon \ge h^2$ to find
\begin{align*}
 h \|D_{\gamma, h} {\rm Def}_{\gamma, h} (I_\T \bu^\varepsilon - \bU^\varepsilon)\|_{L_2(\gamma)} \lesssim  h \varepsilon^{-1/2} \vvvert I_\T \bu^\varepsilon -\bU^\varepsilon \vvvert_{1,h, \varepsilon} \lesssim h \|{\bf f}\|_{L_2(\gamma)}.  
\end{align*}
Relation \eqref{eq:final_approx} follows from \eqref{e:FEM_approx} and \eqref{eq:consistency}.}
\end{proof}

\section{$L_2$ Error Estimates}
\label{sec:L2}

In this section we prove $L_2$ error estimates.  The discrete approximation $\bU^\varepsilon$ does not belong to $\mathcal K^\perp$.
However, we can estimate the $L^{2}$ norm of its orthogonal component
$\| P_{\mathcal K} \bU^\epsilon \|_{L_2(\gamma)}$, 
where we recall that $P_{\mathcal K}$ is the $L_2$ projection onto $\mathcal K$.
This is the subject of the next result. 

\begin{lemma}[Killing Fields Estimates]
    \label{lem:kfields}  \ARD{Assume that $h^2 \le \varepsilon \le 1$.  Then }
    \begin{align}
        \label{eq:PUK}
        \| P_{\mathcal K} \bU^\varepsilon \|_{L_2(\gamma)}
        \lesssim 
         \varepsilon^{-1} h^2 
        \| \bef \|_{L_2(\gamma)}.
    \end{align}
\end{lemma}
\begin{proof}
    From the definition of $P_{\mathcal K}$ and the property $\bu^\varepsilon \in \mathcal K^\perp$, we obtain
    \begin{gather*}
        \| P_{\mathcal K} \bU^\epsilon \|_{L_2(\gamma)} 
        = 
        \sup_{\bv \in \mathcal K, \| \bv \|_{L_2(\gamma)}=1} ( \bU^\varepsilon , \bv) 
        = 
        \sup_{\bv \in \mathcal K, \| \bv \|_{L_2(\gamma)}=1} (\bu^\epsilon-\bU^\varepsilon,  \bv)
        .
    \end{gather*}
    Recalling the BDM interpolant $I_\T$, we write for $\bv \in \mathcal K$
    \begin{equation}\label{e:Killing_proj_to_estim}
    (\bu^\epsilon-\bU^\varepsilon,\bv) = (\bu^\epsilon-\bU^\varepsilon, \bv-I_\T \bv) + (\bu^\epsilon-\bU^\varepsilon, I_\T \bv).
    \end{equation}
    We estimate both terms in the right hand side of the above expression separately and start with the second one.
    Notice that because $\bv$ is a Killing field there holds
    \begin{gather*}
        a(\bu^\varepsilon - \bU^\varepsilon,  \bv) 
        +j(\bu^\varepsilon-\bU^\varepsilon,  \bv)
        = 0.
    \end{gather*}
    Whence, using the definitions of $\bu^\varepsilon$ and $\bU^\varepsilon$  we obtain
    \begin{gather*}
        \varepsilon (\bu^\epsilon-\bU^\varepsilon, I_\T \bv)
        = 
        - a(\bu^\varepsilon - \bU^\varepsilon,  I_\T \bv - \bv) 
        -j(\bu^\varepsilon-\bU^\varepsilon, I_\T \bv- \bv).
    \end{gather*}
    The continuity estimate \ARD{\eqref{dg_continuity}} in conjunction with the finite element approximation estimate \eqref{e:FEM_approx}, the regularity property of the Killing fields \eqref{e:killing_reg}, and the approximation property \eqref{eq:approx} of the BDM interpolant now imply
    \begin{gather*}
        \varepsilon (\bu^\varepsilon-\bU^\varepsilon,I_\T \bv) 
        \lesssim 
        h^2 \| \bef \|_{L_2(\gamma)} \| \bv \|_{H^2(\gamma)}
    \end{gather*}
    and so
    \begin{gather*}
    (\bu^\varepsilon-\bU^\varepsilon, I_\T \bv)  
    \lesssim 
    \varepsilon^{-1} h^2 \| \bef \|_{L_2(\gamma)} \| \bv \|_{H^2(\gamma)}.
    \end{gather*}
    For the first term on the right hand side of \eqref{e:Killing_proj_to_estim}, the finite element approximation estimate \eqref{e:FEM_approx} and the approximation property \eqref{eq:approx} of the BDM interpolant implies as well that
    \begin{align*}
    \begin{aligned}
        (\bu^\epsilon-\bU^\varepsilon, \bv-I_\T \bv) &\le \varepsilon^{-1/2} ( \varepsilon^{1/2} \|\bu^\varepsilon-\bU^\varepsilon\|_{L_2(\gamma)}) \|\bv - I_\T \bv \|_{L_2(\gamma)} 
\\ &         \lesssim  
        \varepsilon^{-\frac 1 2} h^2 \| \bef \|_{L_2(\gamma)} \| \bv \|_{H^2(\gamma)}
        .
        \end{aligned}
    \end{align*}
    Gathering the above two estimates in \eqref{e:Killing_proj_to_estim} and noting that
    \begin{gather*}
        \| \bv \|_{H^2(\gamma)} \lesssim \| \bv \|_{L_2(\gamma)} \qquad \forall \bv \in \mathcal K
    \end{gather*}
    thanks to the norm equivalence property on a finite dimensional space ($\textrm{dim}(\mathcal K) \leq 3$) yield the desired result.
\end{proof}

We now derive an $L_2$-error estimates using a standard duality argument.
\begin{lemma}  \ARD{Assume that $h^2 \le \varepsilon \le 1$.  Then}
\label{lem:basicL2}
    \begin{align}
        \label{eq:basicL2}
        \| \bu^\varepsilon - (\bU^\varepsilon - P_\mathcal K \bU^\varepsilon) \|_{L_2(\gamma)} \lesssim h^2 \| \bef \|_{L_2(\gamma)}.
    \end{align}
\end{lemma}
\begin{proof}
    Define $\bz^{\varepsilon} \in \VT \cap \mathcal K^{\perp}$ as the solution to the elliptic problem
    \begin{gather*}
        a_\varepsilon(\bz^{\varepsilon},\bv)
        = 
        \int_{\gamma}(\bu^\varepsilon - (\bU^\varepsilon - P_\mathcal K \bU^\varepsilon) ) \cdot \bv,
        \qquad
        \forall \bv \in \VT.
    \end{gather*}
    Integrating by parts the strong form
    \begin{gather*}
        -2 \Pi \divergence_\gamma \Deform_\gamma \bu^\varepsilon + \varepsilon \bu^\epsilon 
        = 
        \bu^\varepsilon - (\bU^\varepsilon - P_\mathcal K \bU^\varepsilon)
    \end{gather*}
    (in the sense of distributions) of this relationship yields
    \begin{gather*}
        \begin{aligned}
            \| \bu^\varepsilon & - (\bU^\varepsilon - P_\mathcal K \bU^\varepsilon) \|_{L_2(\gamma)}^2 
            \\ & = a_\varepsilon(\bz^\varepsilon,  \bu^\varepsilon - (\bU^\varepsilon - P_\mathcal K \bU^\varepsilon)) 
            -2 \int_{\Sigma } 
            \{ \Deform_\gamma \bz^\varepsilon \cdot \eta \} \cdot \lbrack \bu^\varepsilon-(\bU^\varepsilon-P_\mathcal{K} \bU^\varepsilon) \rbrack
            \\ & =  a_\varepsilon(\bz^\varepsilon,  \bu^\varepsilon - \bU^\varepsilon) + j(\bz^\varepsilon,  \bu^\varepsilon - \bU^\varepsilon),
        \end{aligned}
    \end{gather*}
    where we used that $P_{\mathcal K} \bU^\varepsilon \in \mathcal K$ and $\bz^\varepsilon \in \HT[2] \cap \mathcal{K}^\perp$ to derive the last equality.
    Now, Galerkin orthogonality together with the continuity estimate \ARD{\eqref{dg_continuity} and the error estimate \eqref{e:FEM_approx}} imply that 
    \begin{align*}
        \| \bu^\varepsilon - ( \bU^\varepsilon - P_\mathcal K \bU^\varepsilon) \|_{L_2(\gamma)}^2 & = \ARD{a_\varepsilon(\bz^\varepsilon -I_\T \bz^\varepsilon,  \bu^\varepsilon - \bU^\varepsilon) + j(\bz^\varepsilon-I_\T \bz^\varepsilon,  \bu^\varepsilon - \bU^\varepsilon)  }
  \\  &\ARD{\lesssim h^2 \|\bu^\varepsilon\|_{H^2(\gamma)} \|\bz^\varepsilon\|_{H^2(\gamma)}.  }
    \end{align*}
   The elliptic regularity estimates \eqref{e:reg_generic} and \eqref{e:elliptic_eps_generic} then yield
    \begin{gather*}
        \| \bu^\varepsilon - (\bU^\varepsilon - P_\mathcal K \bU^\varepsilon )\|_{L_2(\gamma)} \lesssim h^2 \| \bef \|_{L_2(\gamma)},
    \end{gather*}
    which is the desired estimate.
\end{proof}

We now derive two direct corollaries.
\begin{corollary}
\ARD{Assume that $h^2 \le \varepsilon \le 1$.  Then}
    \label{cor:L2}
    \begin{equation}
    \label{L2eps}
    \| \bu^\varepsilon - \bU^\varepsilon \|_{L_2(\gamma)} 
    \lesssim 
\varepsilon^{-1} h^2   
    \| \bef \|_{L_2(\gamma)}.
    \end{equation}
    In particular, when $\varepsilon = h^\alpha$ $(1 \le \alpha \le 2)$ we have
    \begin{equation}
    \label{subopt_L2}
    \| \bu - \bU^\varepsilon \|_{L_2(\gamma)}  \lesssim h^{2-\alpha} \| \bef \|_{L_2(\gamma)}.
    \end{equation}
\end{corollary}
\begin{proof}
    Simply combine  \eqref{eq:PUK} and \eqref{eq:basicL2} to get \eqref{L2eps}.
    For \eqref{subopt_L2}, we  write
    \begin{equation*}
        \begin{split}
            \| \bu - \bU^\varepsilon \|_{L_2(\gamma)} 
            &\leq
            \| \bu - \bu^\varepsilon \|_{L_2(\gamma)} + \| \bu^\varepsilon - \bU^\varepsilon \|_{L_2(\gamma)} 
            \\&\lesssim 
            \varepsilon + ( h^2 + \varepsilon^{-1} h^2 + \varepsilon^{-\frac12} h^2 )  \| \bef \|_{L_2(\gamma)}
            \lesssim 
            h^{2-\alpha} \| \bef \|_{L_2(\gamma)}.
        \end{split}
    \end{equation*}
   \end{proof}

The above corollary establishes that the approximation converges 
but with suboptimal order when $\varepsilon=h^{2-\alpha}$ for $1 \le \alpha <2$, 
but indicates that $\bU^{h^2}$ does not converge to $\bu$ in $L_2$
(which we confirm computationally below).  
However, combining \eqref{eq:consistency} with \eqref{eq:basicL2} yields an optimally convergent result 
when $\varepsilon=h^2$ and Killing fields are removed.
\begin{corollary}
    \label{cor:L2_second}
    If $\varepsilon=h^2$, then
    \begin{align}
        \label{opt_L2} 
        \| \bu - (\bU^\varepsilon - P_{\mathcal K}\bU^\varepsilon ) \|_{L_2(\gamma)} \lesssim  (\varepsilon + h^2) \| \bef \|_{L_2(\gamma)} \lesssim h^2 \| \bef \|_{L_2(\gamma)}.
    \end{align}
\end{corollary}

We finally prove error estimates for the full discrete $H^1$ norm, that is, 
for the total as opposed to the symmetric gradient.
\begin{corollary}
    If $\varepsilon=h$, then
    \begin{gather*}
        \left(\sum_{T\in \T} \|\nabla_{\gamma} (\bu-\bU^\varepsilon)\|_{L_2(T)}^2 \right)^{1/2}\lesssim h \|\bef\|_{L_2(\gamma)}.
    \end{gather*}
    If instead $\varepsilon = h^2$, then
    \begin{gather*}
        \left(\sum_{T\in \T} \|\nabla_{\gamma} (\bu-(\bU^\varepsilon-P_{\mathcal{K}}\bU^\varepsilon))\|_{L_2(T)}^2 \right)^{1/2}\lesssim h \|\bef\|_{L_2(\gamma)}.
    \end{gather*}
\end{corollary}
\begin{proof} 
    In order to prove the first estimate, on $T\in \T$ we write
    $\nabla_{\gamma}(\bu-\bU^\varepsilon)=\nabla_{\gamma} (\bu-I_\T \bu)+ \nabla_{\gamma} (I_\T \bu -\bU^\varepsilon)$. 
    The first of these terms may be directly bounded using the approximation estimate \eqref{eq:approx}.  
    The second may be bounded by applying the Korn-type inequality \eqref{eq:disc_Korn}, 
    adding and subtracting $\bu$ to the result, 
    and then applying the approximation bound \eqref{eq:approx} along with the energy and $L_2$ error bounds \eqref{eq:final_approx} and \eqref{subopt_L2}.  

    To prove the second estimate, 
    we add and subtract $I_\T(\bu -P_{\mathcal{K}} \bU^\varepsilon)$ and proceed essentially as above.  
    The only substantial difference occurs when bounding various norms of $(1-I_\T) P_{\mathcal{K}}\bU^\varepsilon $; 
    here we observe that all norms are equivalent on the finite dimensional space $\mathcal{K}$ and so all norms of $P_{\mathcal{K}}\bU^\varepsilon$ over $\gamma$ are bounded by $\|P_{\mathcal{K}}\bU^\varepsilon\|_{L_2(\gamma)}$, which is in turn bounded by $\|f\|_{L_2(\gamma)}$ via \eqref{eq:PUK}.  
\end{proof}

\section{Filtering out Killing fields}
\label{sec:killing}
Comparing \eqref{opt_L2} and \eqref{subopt_L2}, 
we see that it is possible to obtain a convergent approximation in $L_2$ 
without paying special attention to the Killing fields, 
but it is necessary to explicitly filter them out in order to obtain an optimal $O(h^2)$ convergence rate in the $L_2$ norm or to obtain convergence of the full (unsymmetrized) gradient when $\varepsilon=h^2$. 
If $\gamma$ is known exactly it is sometimes possible to exactly identify the Killing fields.  Such is the case for example for the sphere or for ellipsoids, which we consider in our numerical experiments below.  
However, for more complicated surfaces it might be desirable to automatically identify Killing fields
rather than attempting to compute them manually. 
In addition, it commonly occurs in practice that $\gamma$ is not known exactly.  
Rather one can only assume access to an approximation or family of approximations to $\gamma$.  
For this reason we seek a computational method for identifying Killing fields.  

Our strategy is to approximate the Killing fields via a Stokes eigenvalue problem
and then compute a discrete approximation to $P_{\mathcal{K}} \bU^\varepsilon$ in \eqref{opt_L2}.  
The problem of identifying the Killing fields is however subtle.  
Recall that $0 \le {\rm dim}(\mathcal{K}) \le 3$, with the dimension depending on symmetries of $\gamma$.  
The dimension of $\mathcal{K}$ is thus not stable under small perturbations of $\gamma$ which create or break symmetries.  
This is of particular significance in the context of surface FEM, 
since one typically approximates PDE solutions on a continuous surface $\gamma$ via approximations on a ``nearby'' discrete surface $\Gamma$, which generally does not have the same symmetries as $\gamma$. 
In addition, the nonconforming nature of our FEM causes difficulties when attempting to identify Killing fields.  
Killing fields correspond to zero eigenvalues of a Stokes eigenvalue problem.  
However, there is no reason to expect that the corresponding discrete eigenvalue problem has {\it any} zero eigenvalues, 
whether we compute on $\gamma$ as we have assumed above or a nearby discrete surface $\Gamma$ as is often done in practice.  
Because a ``small'' discrete eigenvalue may correspond to either a Killing field or a mode of a near-symmetric surface, 
in general it is not possible to immediately identify from a discrete eigenvalue approximation which modes correspond to Killing fields and which do not.   \ARD{Identifying ${\rm dim}(\mathcal{K})$ numerically is thus an ill-posed problem, so we take a more circuitous route to filtering out Killing fields and the results that we achieve are relatively modest.  We show that we can choose an approximation $\mathcal{K}_h$ to $\mathcal{K}$ such that $\bU^\varepsilon-P_{\mathcal{K}_h}\bU^\varepsilon$ is a reasonable approximation to $\bu$ in the $L_2$ norm even if ${\rm dim}(\mathcal{K})$ has not been correctly identified.  Asymptotically $\mathcal{K}_h$ is guaranteed to correctly approximate $\mathcal{K}$, but it is not known to the user whether this is the case in a given computation.
}

\ARD{We also address the practical implications of removing Killing fields from the discrete solution and from the right hand side ${\bf f}$.   Computational solutions of free boundary problems involving surface Stokes fluid models indicate that rotational symmetries may be destroyed \cite{BGN16} or produced \cite{BGN17} as the surface evolves, depending on the physical situation assumed.   The implication of the latter computation is that Killing fields may be present in physically relevant computations even if symmetries are not initially present.  If rotational modes are present in a solution to the corresponding stationary Stokes problem, they might obscure other relevant features of the equilibrium flow and so it is desirable to remove them.  It may be preferable in some cases for a user to manually determine ${\rm dim}(\mathcal{K})$ by using eigenvalue data (as described below) and possibly visual inspection of symmetries, but our  algorithm provides more robust guarantees of solution quality.  These strategies could also be used in combination.  }

\ARD{We also discuss the possibility that the forcing function ${\bf f} \not\perp \mathcal{K}$.  The extent to which this situation is physically relevant is unclear.  If a Killing mode is present in a forcing function ${\bf f}$ driving a fluid flow on a radially symmetric surface, the surface would rotate with increasing velocity in the direction of the Killing mode.  The condition ${\bf f} \perp \mathcal{K}$ thus appears to be physically as well as mathematically reasonable.  Nonetheless we briefly address below the question of approximating ${\bf f}-P_{\mathcal{K}} {\bf f}$ in finite element approximations to the stationary surface Stokes problem.  
}

\subsection{A Stokes eigenvalue problem}  

Consider first the Stokes eigenvalue problem:  Find $(\bu, \lambda) \in \VT \times \overline{\mathbb{R}^+}$ such that
\begin{align}
\label{cont_eigen}
a(\bu, \bv) = \lambda (\bu, \bv), ~~\bv \in  \VT.  
\end{align} 
There exists then a sequence of eigenvalues $0 \le \lambda_1 \le \lambda_2 \leq \dots$ 
increasing to $\infty$, and corresponding $L_2$-orthonormal basis of eigenfunctions $\{\bu_i\}_{i=1}^\infty$.
It is easily seen that $\mathcal{K}$ is the space of eigenfunctions corresponding to $\lambda = 0$.  

We similarly consider the discrete eigenvalue problem:  
Find $(\bU, \Lambda) \in \VT[\T] \times \mathbb{R}$ such that
\begin{align}
\label{disc_eigen}
a(\bU, \bV)+j(\bU, \bV) = \Lambda (\bU, \bV), ~~\bV \in \VT[\mathcal{T}].
\end{align} 
We denote the discrete eigenvalues by $\{\Lambda_i\}_{i=1}^N$, 
and an orthonormal basis of eigenvectors by $\{\bU_i\}_{i=1}^N$.  
Note that we have given no guarantee that this system is positive, so there may in principle be negative eigenvalues.  
Also, it is possible for both the continuous and discrete problems to pose an equivalent eigenvalue problem based on the full Stokes system \eqref{weakform}.  
We do so in computational practice due to the difficulty of identifying the divergence-free subspace $\VT[\T]$.  
For theoretical purposes, however, the reduced systems here are more convenient.  
We refer to \cite{GK18} for a similar method for the Euclidean Stokes eigenvalue problem.  



Below we give a robust method which asymptotically filters out Killing fields and thus yields an asymptotic $O(h^2)$ convergence rate in $L_2$.
The basic assumption for our method is that $|\lambda_i-\Lambda_i| \lesssim h^2$ and $\|\bu_i-\bU_i\|_{L_2(\gamma)} \lesssim h^2$.  
We prove this assumption for the discrete approximation \eqref{disc_eigen} posed over the continuous surface $\gamma$.  
Our framework is however still valid in the more practically relevant case where the discrete problem is posed over a nearby discrete surface $\Gamma$, 
provided that the additional geometric consistency error is also of order $h^2$.   
Numerical evidence presented below indicates that this is the case when $\Gamma$ is taken to be a polyhedron with triangular faces as above.

We first prove convergence of eigenvalues and eigenvectors.
\begin{theorem}
    Let $\lambda_i$ be an eigenvalue and  
    let $P_{\lambda_i}$ denote the projection onto 
       the eigenspace spanned by the eigenfunctions 
    corresponding to $\lambda_i$.  \ARD{Let $(\Lambda_i, \bU_i)$ be a corresponding discrete eigenpair.}  
    Then
    \begin{align}
        \label{eigen:errors}
        \|\bU-P_{\lambda} \bU\|_{L_2(\gamma)} 
        + |\lambda-\Lambda| \lesssim h^2.
    \end{align}
\end{theorem}
\begin{proof}
    We first reformulate the eigenvalue problem in order to ensure that the discrete system is definite.    
    Consider the problem:  Find $(\bu, \tilde{\lambda}) \in \VT \times \mathbb{R}^+$ such that 
    \begin{align}
        \label{cont_eigen_stable}
        \begin{aligned}
            a_1(\bu, \bv)  = \tilde{\lambda} (\bu, \bv), ~~\forall \bv \in \VT.
        \end{aligned}
    \end{align} 
    It is easy to see that that eigenpairs of \eqref{cont_eigen_stable} correspond to eigenpairs of \eqref{cont_eigen}
    with the identification $\lambda=\tilde{\lambda}-1$.  
    The eigenfunctions are precisely the same.  
    Similarly let $(\bU, \tilde{\Lambda}) \in \VT[\T] \times \mathbb{R}^+$ satisfy
    \begin{align}
        \label{disc_eigen_stable}
        a_1(\bU, \bV)+ j(\bU, \bV)=\tilde{\Lambda} (\bU, \bV), ~~~\forall \bv \in \VT[\T].
    \end{align}
    As with the continuous problem, we have $\Lambda=\tilde{\Lambda}-1$. 

    We next apply the Babu\v{s}ka-Osborn theory to obtain convergence of the eigenvalue problem.  Let first $\bX$ be the divergence-free subspace of $H(\divergence_\gamma; \gamma)$.  The native $H(\divergence)$ norm for this space reduces to the $L_2$ norm, and it is easy to show that $\bX$ is closed in $H(\divergence_\gamma; \gamma)$.  Let $T:\bX \rightarrow \bX$ be given by  
    \begin{align}
    \label{T_def}
        a_1(T \bu, \bv)= (\bu, \bv), ~~~\bv \in \VT.
    \end{align}
    This problem is well-posed, and because the norm on $\bX$ reduces to the $L_2$ norm we also easily obtain compactness of $T$ by standard arguments.  In particular, the image of a bounded set in $\bX$ under $T$ is bounded in $\HT$ and thus precompact in $L_2$, and therefore in $\bX$ since the norms are equivalent and $\bX$ is closed.
    The corresponding discrete solution operator $T_h:\bX \rightarrow \VT[\T] \subset \bX$ is given by 
    \begin{align}
        a_1(T_h \bU, \bV)+j(T_h \bU, \bV)=(\bU, \bV), ~~~\bV \in \VT[\T].
    \end{align}
    This problem is well-posed according to Lemma~\ref{lem:well-posed}.  \ARD{It remains to prove the estimate}
    \begin{align}
    \label{tminusth}
        \|(T-T_h) \bu\|_{\bX} \lesssim h^2 \|\bu\|_{\bX}, ~~~ \bu \in \bX,
    \end{align}
\ARD{which is \eqref{L2eps} (with $\varepsilon=1$) in the case $\bu \perp \mathcal{K}$.  In order to prove this estimate in the general case we first prove an $H^2$ regularity bound.  First set $\bu=\bu_{\mathcal{K}} + \bu_\perp$ with $\bu_{\mathcal{K}} \in \mathcal{K}$ and $\bu_\perp \perp \mathcal{K}$.  $\|T \bu_\perp\| \lesssim \|\bu_\perp\|_{L_2(\gamma)}$ follows from \eqref{e:reg_u_eps}.  We also may easily compute that $T \bu_\mathcal{K}=\bu_{\mathcal{K}}$ by testing \eqref{T_def} first with $\bv \in \mathcal{K}$ and then with $\bv \in \mathcal{K}_\perp$.  Then $\|T \bu_{\mathcal{K}}\|_{H^2(\gamma)} = \|\bu_{\mathcal{K}}\|_{H^2(\gamma)} \lesssim \|\bu_{\mathcal{K}}\|_{L_2(\gamma)}$, since $\mathcal{K}$ is a finite-dimensional space.  Putting these estimates together yields  $\|T\bu \|_{H^2(\gamma)} \lesssim \|\bu \|_{L_2(\gamma)}$.   By combining this estimate with coercivity of $a_1(\cdot, \cdot)$ over ${\bf VT}(\T)$  (cf. Lemma~\ref{lem:well-posed}) and approximation properties, we obtain the estimate \eqref{tminusth} using standard techniques as in Lemma \ref{lem:basicL2}.  }

    With this framework in hand, the Babu\v{s}ka-Osborn theory yields the desired result; we refer to \cite[Section 11]{Bof10} for details of this theory in the context of nonconforming approximations such as those we consider here.  
\end{proof}


\subsection{Filtering out Killing fields}
\label{subsec:Kfilt}

We denote by $\bu_i$ resp. $\bU_i$ the eigenfunction associated with
$0 \leq \lambda_1 \leq \lambda_2 \leq \dots$ resp. $0 \leq \Lambda_1 \leq \Lambda_2 \leq \dots$.
We also recall that $\bU^\varepsilon$ (where we shall take either $\varepsilon=h^\alpha$ with $1 \le \alpha <2$ or $\varepsilon=h^2$) and the eigenfunctions $\{\bU_i\}_{i=1}^N$, with $N:= \textrm{dim}(\VT[\T])$, are all strongly divergence free.  Using \eqref{FE_def} along with the discrete eigenvalue problem \eqref{disc_eigen} thus yields
\begin{align}
    \label{eigenexpansion}
    \bU^\varepsilon=\sum_{i=1}^N \frac{({\bf f}, \bU_i)}{\Lambda_i+\epsilon} \bU_i.
\end{align}

We recall that $\mathcal{K}={\rm span}_{1 \le i \le j_{\mathcal{K}}} \{\bu_i\}$, where $0 \le j_{\mathcal{K}} \le 3$ depends on $\gamma$.  
Let $J_\mathcal{K}=\{\{1\}, \{2\}, \{3\}, \{1,2\}, \{1,3\}, \{2,3\}, \{1,2,3\}\}$, with elements denoted by $J$.  
We may thus also write that $\mathcal{K}={\rm span}_{i \in J} \{\bu_i\}$ for some $J \in J_{\mathcal{K}}$.  
Our goal is to use the discrete eigenvalues $\{\Lambda_i\}_{i=1}^N$ 
in order to choose which discrete eigenfunctions correspond to Killing fields, 
and which we should therefore remove from our solution.  
Given $J \in J_{\mathcal{K}}$, let $\mathcal{K}_{J,h}={\rm span}_{j \in J} \{\bU_j\}$.  
Let $1 \le \alpha <2$.  We prove below that choosing a set of putative discrete Killing fields $\mathcal{K}_{J,h}$ so that $\|\bU^{h^\alpha}-(\bU^{h^2}-P_{\mathcal{K}_{J,h}} \bU^{h^2})\|_{L_2(\gamma)}$ is minimized leads an approximation $\bU^{h^2}-P_{\mathcal{K}_h} \bU^{h^2}$ which asymptotically converges with order $h^2$ to $\bu$. 
We start by providing a characterization of this minimal set.

\begin{lemma}\label{l:J*}
    Assume that $ h\leq 1$ and $1 \le \alpha <2$. We have that
    \begin{equation}\label{e:J*}
    J^*_h:= \left\lbrace i \ : \  \Lambda_i \leq h^\alpha -2h^2 \right\rbrace
    \end{equation}
    satisfies
    \begin{align}
    \label{eq120}
        J^*_h= \mathop{\mathrm{argmin}}_{J  \in J_{\mathcal{K}}} \|\bU^{h^\alpha}-(\bU^{h^2}-P_{\mathcal{K}_{J,h}}\bU^{h^2})\|_{L_2(\gamma)}.
    \end{align}
    Moreover, there exists $h_0>0$ such that for $h \leq h_0$ 
    we have that $\bu_i$ is a Killing field if and only if $i \in J^*_h$.
\end{lemma}

\begin{proof}
      Let $J \in J_{\mathcal{K}}$.  Using \eqref{eigenexpansion} with $\varepsilon=h^\alpha$ and $\varepsilon=h^2$, we compute
    \begin{align*}
    \begin{aligned}
        \|\bU^{h^\alpha}&-(\bU^{h^2}-P_{\mathcal{K}_{J,h}}\bU^{h^2})\|_{L_2(\gamma)}^2
        \\ & =
        \sum_{i \in J} 
        \frac{(\bef,\bU_i)^2}{(\Lambda_i+h^\alpha)^2}
        + 
        \sum_{1 \le i \le N, i \notin J} 
        \frac{(\bef, \bU_i)^2(h^\alpha-h^2)^2}{(\Lambda_i+h^\alpha)^2(\Lambda_i+h^2)^2}.
   \end{aligned}
    \end{align*}
\ARD{Assuming briefly that $({\bf f}, \bU^i) \neq 0$ for $1 \le i \le 3$,}  the index set $J$ yields the minimal set in \eqref{eq120} if and only if 
    \begin{gather}\label{e:threshold}
        \frac{(\bef,\bU_i)^2}{(\Lambda_i+h^\alpha)^2}
        \le 
        \frac{(\bef, \bU_i)^2(h^\alpha-h^2)^2}{(\Lambda_i+h^\alpha)^2(\Lambda_i+h^2)^2},
    \end{gather}
    for all $i \in J$.  Equivalently, the minimal  set is such that
    \begin{align}
\label{eig_condition}
        \Lambda_i \le h^\alpha-2h^2, ~~i \in J,
    \end{align}
    which proves \ARD{the assertion that \eqref{e:J*} yields \eqref{eq120} when $({\bf f}, \bU^i) \neq 0$, $1 \le i \le 3$.  If $({\bf f}, \bU^i)=0$ for some $1 \le i \le 3$, then inclusion or exclusion of the $i$-th mode in the eigenexpansion does not affect whether \eqref{eq120} holds, so employing \eqref{eig_condition} still yields \eqref{eq120}. }
    The eigenvalue error estimate \eqref{eigen:errors} and the fact that $\bu_i$ is a Killing field $\iff \lambda_i=0$ establish that $\bu_i$ is a Killing field  if and only if $i \in J^*_h$ for $h\leq h_0$.
\end{proof}

We can now derive the optimal convergence of the quantity $\|\bU^{h^\alpha}-(\bU^{h^2}-P_{\mathcal{K}_{J,h}}\bU^{h^2})\|_{L_2(\gamma)}$.
\begin{theorem}
    Let $J^*_h$, $h_0$, and $1 \le \alpha <2$ be as in Lemma~\ref{l:J*} and set $\mathcal K^*_h := \mathcal{K}_{J^*_h,h}$.
    For all $h$ we have
    \begin{align}
        \|\bu-(\bU^{h^2}-P_{\mathcal K^*_h}\bU^{h^2})\|_{L_2(\gamma)} \lesssim h^{2-\alpha}.
    \end{align}
    Furthermore, whenever $h\leq h_0$ we have
    \begin{align}
        \|\bu-(\bU^{h^2}-P_{\mathcal K^*_h}\bU^{h^2})\|_{L_2(\gamma)} \lesssim h^2
        .
    \end{align}
\end{theorem}

\begin{proof}
    We start with the case $h \le h_0$.
    From Lemma~\ref{l:J*} we deduce that $\mathcal{K}^*_h$ has the same dimension as $\mathcal{K}$, and per the eigenfunction error estimate  \eqref{eigen:errors} we have 
    \begin{gather*}
       \ARD{ \|P_{\mathcal{K}_h^*} \bU^{h^2} -P_\mathcal{K} \bU^{h^2}\|_{L_2(\gamma)} \lesssim h^2.}
    \end{gather*} 
Combining this observation with the $L_2$ error estimate \eqref{opt_L2} yields that for $h \le h_0$, 
    \begin{gather*}
        \begin{aligned}
            \|\bu&-(\bU^{h^2}-P_{\mathcal{K}^*_h} \bU^{h^2})\|_{L_2(\gamma)} 
            \\ &  \le \|\bu-(\bU^{h^2}-P_{\mathcal{K}} \bU^{h^2})\|_{L_2(\gamma)}+\|P_{\mathcal{K}} \bU^{h^2}-P_{\mathcal{K}^*_h} \bU^{h^2}\|_{L_2(\gamma)} \lesssim h^2.
        \end{aligned}
    \end{gather*}

    For $h > h_0$, we let $J \in J_{\mathcal{K}}$ be the index set corresponding to $\mathcal{K}$.  We then calculate that for $J^*_h$
    \begin{align}
        \begin{aligned}
            \|\bu-(\bU^{h^2}&-P_{\mathcal{K}^*_h}\bU^{h^2})\|_{L_2(\gamma)} 
            \\ &  \le \|\bu-\bU^{h^\alpha}\|_{L_2(\gamma)} + \|\bU^{h^\alpha}-(\bU^{h^2}-P_{\mathcal{K}^*_h}\bU^{h^2})\|_{L_2(\gamma)} 
            \\ & \lesssim h^{2-\alpha} + \|\bU^{h^\alpha}-(\bU^{h^2}-P_{\mathcal{K}_{J,h}}\bU^{h^2})\|_{L_2(\gamma)}
            \\ & \lesssim h^{2-\alpha} + \|\bU^{h^\alpha} - \bu\|_{L_2(\gamma)} + \|\bu - (\bU^{h^2}-P_\mathcal{K} \bU^{h^2})\|_{L_2(\gamma)} 
            \\ &\qquad \qquad+ \|P_\mathcal{K} \bU^{h^2}- P_{\mathcal{K}_{J,h}}\bU^{h^2}\|_{L_2(\gamma)} 
            \\ & \lesssim h^{2-\alpha} + h^{2-\alpha} + h^2 + h^2  \lesssim h.
        \end{aligned}
    \end{align}
    Here we have used \eqref{subopt_L2} and \eqref{opt_L2} along with \eqref{eigen:errors}.
    This ends the proof.  
\end{proof}

Lemma~\ref{l:J*} shows that the discrete eigenfunctions that are filtered out via the minimization problem are precisely those corresponding to eigenvalues satisfying $\Lambda_i \le h^\alpha-2h^2$.  We thus may apply this criterion directly without computing $\bU^{h^\alpha}$.  
Still, we need to compute $\bU^{h^2}$ and the first three eigenvalue/eigenfunction pairs in \eqref{disc_eigen}.  We sum this fact up in the following corollary.

\begin{corollary}
    \label{cor:Kfields}
    Let $J_h:= \{ j \in {1,2,3} \ : \Lambda_j \leq h^\alpha -2h^2\}$, with $1 \le \alpha <2$.  Then 
    \begin{gather*}
        \|\bu-(\bU^{h^2} -P_{\mathcal{K}_{J_h,h}} \bU^{h^2})\|_{L_2(\Gamma)} \lesssim h^{2-\alpha},
    \end{gather*}
    and there exists $h_0$ such that if $h \le h_0$, then 
    \begin{gather*}
        \|\bu-(\bU^{h^2} -P_{\mathcal{K}_{J_h,h}} \bU^{h^2})\|_{L_2(\Gamma)} \lesssim h^2.
    \end{gather*}
\end{corollary}

\ARD{\subsection{Filtering Killing fields from ${\bf f}$} \label{sub_kfieldsfromf}  Here we discuss options when in contrast to our previous assumption ${\bf f} \not\perp \mathcal{K}$.   In this case removal of Killing fields is essential.  The continuous problem is not well-posed, and in our $\varepsilon$-perturbed discrete formulation this illposedness is manifested by amplification of Killing modes by a factor of $\frac{1}{\Lambda_i+\varepsilon}$ as seen in \eqref{eigenexpansion}.  Since $\Lambda_i=O(h^2)$ these modes will dominate as $h, \varepsilon \rightarrow 0$ if not removed.}

\ARD{We first obtain the following lemma when ${\rm dim}(\mathcal{K})$ is known.
\begin{lemma}
\label{lem:fminuspkhf}
Let $\bU^\varepsilon$ solve the finite element system \eqref{FE_def} with $\varepsilon=h^2$.  Let also ${\mathcal{K}_h}$ be the span of the first ${\rm dim}(\mathcal{K})$ discrete eigenfunctions.  Then for $h$ small enough,
\begin{align}
\label{l2est_pkh}
\|\bu-(\bU^\varepsilon - P_{\mathcal{K}_h} \bU^\varepsilon)\|_{L_2(\gamma)} \lesssim h^2 \|{\bf f}\|_{L_2(\gamma)},
\end{align}
independent of whether ${\bf f}$ is free of of Killing fields or not.
\end{lemma}
\begin{proof}
Let $\bW  \in {\bf VT}(\T)$ solve 
$$a_\varepsilon(\bW, \bV)+j(\bW, \bV)=\int_\gamma ({\bf f}-P_{\mathcal{K}} {\bf f}) \cdot \bV, ~~\bV \in {\bf VT}(\T),$$
and $\bY \in {\bf VT}(\T)$ solve
$$a_\varepsilon(\bY, \bV)+j(\bY, \bV)=\int_\gamma ({\bf f}-P_{\mathcal{K}_h} {\bf f}) \cdot \bV, ~~\bV \in {\bf VT}(\T).$$
From \eqref{eigenexpansion} we have $\bY=\bU^\varepsilon-P_{\mathcal{K}_h} \bU^\varepsilon$.  In addition, \eqref{opt_L2} and \eqref{eigen:errors} yield $\|\bu-(\bW-P_{\mathcal{K}_h} \bW)\|_{L_2(\gamma)} \lesssim h^2 \|{\bf f}\|_{L_2(\gamma)}$.  Thus by the triangle inequality the lemma will be proved once we show that $\|\bY-(\bW-P_{\mathcal{K}_h} \bW)\|_{L_2(\gamma)} \lesssim h^2 \|{\bf f}\|_{L_2(\gamma)}$.  
Using the discrete eigenexpansion \eqref{eigenexpansion} and $P_{\mathcal{K}_h} \bY=0$, we compute that
$$\bY-(\bW-P_{\mathcal{K}_h} \bW) = \bY-\bW-P_{\mathcal{K}_h} (\bY-\bW)= \sum_{i={\rm dim}(\mathcal{K})+1}^N \frac{(P_{\mathcal{K}_h}{\bf f}-P_{\mathcal{K}} {\bf f}, \bU^i)}{\Lambda_i+ \varepsilon} \bU^i.$$
$\Lambda_i \le \Lambda_{i+1}$ and $\lambda_{{\rm dim}(\mathcal{K})+1} >0$, so \eqref{eigen:errors} yields for $h$ small enough that
$$\|\bY-(\bW-P_{\mathcal{K}_h} \bW)\|_{L_2(\gamma)}  \le \frac{1}{\Lambda_{{\rm dim}(\mathcal{K})+1}+ \varepsilon} \|P_{\mathcal{K}_h}{\bf f}-P_{\mathcal{K}}{\bf f}\|_{L_2(\gamma)} \lesssim \frac{h^2}{\lambda_{{\rm dim}(\mathcal{K})+1}}.$$
This completes the proof.  
\end{proof}
}

\ARD{We next give a procedure for filtering Killing fields from $f$ when ${\rm dim}(\mathcal{K})$ is unknown.  If ${\bf f} \not\perp \mathcal{K}$, then solving $a_\varepsilon(\bu_{\bf f}, \bv)=({\bf f}, \bv)$ for ${\bu}_{\bf f} \in {\bf VT}(\gamma)$ yields $\varepsilon P_{\mathcal{K}} \bu_{\bf}=P_{\mathcal{K}} {\bf f}$ and $\|\varepsilon \bu_{\bf f}-P_{\mathcal{K}} {\bf f}\|_{L_2(\gamma)} \lesssim \varepsilon$.  The discrete counterpart is to find $\bU_{\bf f} \in {\bf VT}(\T)$ such that 
\begin{align}
\label{eq:pkfapprox}
a_\varepsilon(\bU_{\bf f}, \bV) + j(\bU_{\bf f}, \bV)=({\bf f}, \bV), ~~\bV \in {\bf VT}(\T).
\end{align}
It can then be shown that
\begin{align}
\|P_{\mathcal{K}}{\bf f}-\varepsilon{\bU_{\bf f}}\|_{L_2(\gamma)} \lesssim h^2 + \varepsilon^{-1} h^2 + \varepsilon.
\end{align}
The function $\bU_{{\bf f}}$ computed with $\varepsilon=h^\alpha$, $0<\alpha<2$, is thus a convergent approximation to $P_{\mathcal{K}} {\bf f}$ which may be obtained without any knowledge of ${\rm dim}(\mathcal{K})$.
 }

\ARD{Let now $\bW \in {\bf VT}(\T)$ solve
\begin{align}
\label{eq100}
a_\varepsilon(\bW, \bV)+j(\bW, \bV)=({\bf f}-\varepsilon \bU_{\bf f}, \bV), ~~\bV \in {\bf VT}(\T).
\end{align} 
Then
\begin{align}
\label{eq101}
 \bW=\sum_{i=1}^N \frac{1}{\Lambda_i+\varepsilon}\left (1-\frac{\varepsilon}{\Lambda_i+\varepsilon} \right) ({\bf f}, \bU^i) \bU^i= \sum_{i=1}^N \frac{\Lambda_i}{(\Lambda_i+\varepsilon)^2} ({\bf f}, \bU^i) \bU^i .
\end{align}
Choosing $\epsilon = h^{2/3}$ yields
\begin{align}
\|\bu - \bW\|_{L_2(\gamma)} \lesssim h^{2/3}.  
\end{align}
Here we have chosen $\varepsilon$ in \eqref{eq:pkfapprox} and \eqref{eq100} to be the same.  }

\ARD{We next leverage the robustly convergent but suboptimal approximation $\bW$ to achieve a preasymptotically convergent and asymptotically optimal approximation.  Solving
\begin{align}
a_\varepsilon (\bU^{h^2}, \bV) + j(\bU^{h^2}, \bV)=({\bf f}, \bV), ~~\bV \in {\bf VT}(\gamma)
\end{align}
yields
\begin{align}
\label{eq110}
\bU^{h^2} = \sum_{i=1}^N \frac{1}{\Lambda_i+h^2} ({\bf f}, \bU^i) \bU^i.
\end{align}
Recall from \eqref{l2est_pkh} that we will obtain an $O(h^2)$ approximation to $\bu$ in $L_2$ if we correctly identify and remove the discrete modes corresponding to Killing fields.  Using the same idea as in Lemma \ref{l:J*}, we consider $\bU^i$ to be a discrete Killing mode and remove the corresponding term from the solution \eqref{eq110} if doing so leads to better fidelity to the robustly convergent solution \eqref{eq101}.   Applying this idea term by term to the potential Killing modes in \eqref{eq101} and \eqref{eq110} ($1 \le i \le 3$), we consider $\bU^i$ to be a discrete Killing mode if 
\begin{align}
\label{eq111}
\left |\frac{\Lambda_i}{(\Lambda_i+h^{2/3})^2}-0 \right | \le \left |\frac{\Lambda_i}{(\Lambda_i+h^{2/3})^2}-\frac{1}{\Lambda_i+h^2} \right | = \frac{1}{\Lambda_i+h^2}
-\frac{\Lambda_i}{(\Lambda_i+h^{2/3})^2};
\end{align} 
compare with \eqref{e:threshold}.
Rearranging this inequality, we define a putative discrete Killing space as
\begin{align}
\label{eq112}
\mathcal{K}_h= {\rm span}\{\bU^i: \Lambda_i+2 \Lambda^i(h^2-h^{2/3}) \le h^{4/3}\}.
\end{align}
With this definition, we have that $\|\bu-(\bU^{h^2} -P_{\mathcal{K}_h} \bU^{h^2})\|_{L_2(\gamma)} \lesssim h^{2/3}$ for all $h$, and for $h$ sufficiently small $\|\bu-(\bU^{h^2} -P_{\mathcal{K}_h} \bU^{h^2})\|_{L_2(\gamma)} \lesssim h^{2}$.  
}

\section{Numerical experiments}
\label{sec:numerics}

We illustrate our findings by numerical experiments. 
We shall take $\gamma$ to be either a sphere or a proper ellipsoid
that is defined by the equation $x^2+y^2+\frac{z^2}{c^2} = 1$ with some parameter $c$.
The case $c=1$ yields the unit sphere,
where we have ${\rm dim}(\mathcal{K})=3$ and an $L_2$-orthogonal (but not normalized) basis for $\mathcal{K}$ 
is given by $\bk_1=(y,-x,0)^T$, $\bk_2=(z, 0, -x)^T$, and $\bk_3=(0, z, -y)^T$.  
The case $c \neq 1$ yields an ellipsoid whose $z$-axis is the only axis of rotational symmetry, 
and so ${\rm dim}(\mathcal{K}) = 1$ and $\mathcal{K} = \{a(y, -x, 0)^T: a \in \mathbb{R} \}$.

Following \cite{OQRY18}, we let $\bu(x,y,z)=\Pi(-z^2, x, y)^T$ for all choices of $c$ and set $p(x,y,z)=xy^3+z$.  
Here ${\rm div}_\gamma \bu \neq 0$, 
so our test solution does not directly fit into the framework described above.  
We still expect all error estimates to hold as stated, and our numerical experiments indicate that they do.  
The algorithm and analysis of Section~\ref{subsec:Kfilt} also remain valid in this setting, 
with appropriate modifications to the analysis since $\bU^\varepsilon$ is no longer divergence-free. 
For instance, the eigenfunction expansion \eqref{eigenexpansion} may be applied to the $L_2$ projection of $\bU^\varepsilon$ onto the divergence-free discrete subspace $\VT[\T]$ and the remaining analysis carried out much the same as before.

\subsection{Implementation}

\ARD{Implementation was carried out within the iFEM MATLAB library \cite{Ch09PP}.}  All computations below were carried out with the finite element method 
defined on a discrete polyhedral approximation $\Gamma$ to $\gamma$
(as described in Subsection~\ref{subsec:discrete_surf}), 
not on the exact surface $\gamma$ as we have assumed above.  \ARD{
The piecewise affine velocity space ${\bf XT}(\bMT)$ and piecewise constant pressure space $\mathbb{V}(\bMT)$ were employed and all derivatives and integrals were computed over the discrete surface $\Gamma$.  A standard BDM basis was defined on a Euclidean reference triangle and mapped via elementwise Piola transform to the discrete surface $\Gamma$.    Further transforming these discrete spaces to the continuous surface $\gamma$ as described above could be easily accomplished so long as the user has access to the closest point map $\bP_d$ and related information used in the definition of the Piola transform; cf. \cite{DD07} for discussion of practical calculation of these quantities.  Implementation for higher-order BDM spaces would be conceptually very similar.  The main difference is that typically computations involving higher order elements are carried out on a polynomial surface approximation of higher degree rather than on the affine surface $\Gamma$, which modestly complicates implementation by for example requiring computation of non-piecewise constant Jacobians.   }

\ARD{Implementation of the algorithms for filtering Killing fields presented in Section \ref{sec:killing} is straightforward.  The discrete eigenproblem \eqref{disc_eigen_stable} is solved for the first three eigenpairs $(\Lambda_i, \bU^i)$  using a standard algorithm.  $\mathcal{K}_h$ is taken as the span of the first ${\rm dim}(\mathcal{K})$ discrete eigenfunctions $\bU^i$ if ${\rm dim}(\mathcal{K})$ is given.  Otherwise the condition \eqref{eig_condition} or \eqref{eq111} (depending on whether ${\bf f} \perp \mathcal{K}$ or ${\bf f} \not\perp \mathcal{K}$) is checked directly for each $1 \le i \le 3$, and $\mathcal{K}_h$ is defined accordingly.  In all cases it is easy to compute $\bU^\varepsilon - P_{\mathcal{K}_h} \bU^\varepsilon$ once $\mathcal{K}_h$ has been defined.  
}

\subsection{Geometric errors}
Approximating $\gamma$ by $\Gamma$ as described in the preceding subsection induces an additional consistency error in the finite element method, 
often called a \emph{geometric error}.  
For the Laplace-Beltrami (surface scalar Laplace) problem, the geometric error has been shown to be $O(h^2)$ for a wide range of finite element methods (including discontinuous Galerkin methods) and error notions (including energy and $L_2$ norms and eigenvalues) \cite{DMS13, ADMSSV15, CD16, BDO18, bonito2019finite}.  
The situation is very different for surface vector Laplace-type problems such as the Stokes equation.   \ARD{The papers \cite{JR20, JORZPP} contain error analysis for trace surface finite element method for surface Stokes and vector Laplace problems.  The lowest-order algorithms in these papers} employ a linear surface approximation $\Gamma$ as we do, but a higher-order ($O(h^2)$) convergent approximation to the normal $\bnu$ is \ARD{required} in order to maintain optimal-order $O(h^2)$ geometric errors.  
In contrast, in our experiments below we consistently observed an $O(h^2)$ geometric error while employing a piecewise linear surface approximation and corresponding natural $O(h)$ piecewise constant normal approximation. 
A theoretical explanation for these observed rates is not clear and bears further investigation.

\subsection{Eigenvalue computation and Killing fields}

Before discussing the finite element method for the Stokes equations,
we illustrate how the geometry of $\gamma$ 
affects the algorithm of Section~\ref{subsec:Kfilt} for filtering out Killing fields. 
We approximated the first three eigenvalues of the Stokes problem \eqref{cont_eigen} by solving \eqref{disc_eigen}
for different values of $c$; see Table~\ref{table_eigs}.  In all cases we have $\lambda_1 = 0$ and $\lambda_2=\lambda_3$ due to symmetries of the ellipsoids under consideration. 
Approximations to $\lambda_2=\lambda_3$ were computed by solving the discrete eigenvalue problem on multiple mesh levels and extrapolating assuming that $|\lambda_i-\Lambda_i| \lesssim Ch^2$, cf. \eqref{eigen:errors}.   
Note for example that when $c=1.1$, the second eigenvalue is nonzero but small.
Absent a priori information, significant resolution of $\gamma$ is necessary
in order to confidently determine ${\rm dim}(\mathcal{K})$.  
\begin{table}[h!]
\begin{center}
\begin{tabular}{|c|c|c|c|c|}
\hline $c$ & $1$ & $1.1$ & $1.25$ & $2$
\\ \hline $\lambda_2=\lambda_3$ & $0$ & $0.0096$ & $0.051$ & 0.40
\\ \hline  
\end{tabular}
\caption{Approximations to the second and third eigenvalues of the Stokes problem on the ellipsoids $x^2+y^2+c^{-2} z^2 = 1$.} 
\label{table_eigs}
\end{center}
\end{table}

\subsection{Basic error behavior} 
In this section we illustrate the basic energy and $L_2$ error estimates of Sections~\ref{subsec:basic_estimates} and Section~\ref{sec:L2}. 
Here we take $\gamma$ to be the sphere ($c=1$) and the ellipsoid given by various values of $c$.  \ARD{Since analytical expressions for the Killing fields are known, we filter them out exactly as needed by computing $\bU^\varepsilon-P_{\mathcal{K}} \bU^\varepsilon$.}  
In Figure~\ref{fig1} we display the convergence of the $L_2$ norm of the error in the deformation gradient 
(which is the practically important part of the energy norm bounded in \eqref{eq:final_approx}).  

The expected $O(h)$ convergence rate is immediately apparent when $\varepsilon=h^2$.  
The behavior is more subtle when $\varepsilon=h$,
where we see consistent $O(h)$ convergence when $c=1$ (for the sphere), 
an initial period of suboptimal convergence and then clear $O(h)$ convergence asymptotically when $c=1.25$, 
and a convergence history that has not yet reached the asymptotic $O(h)$ range when $c=1.1$.  
This behavior may be understood by combining the data in Table~\ref{table_eigs} with the eigenfunction expansion \eqref{eigenexpansion}. 
The coefficients in the eigenexpansion would be $\frac{({\bf f}, \bU_i)}{\Lambda_i}$
if no perturbation were used ($\varepsilon = 0$), 
and these coefficients would yield the best approximation to $\bu$ assuming $\lambda_i \neq 0$. 
The difference in the perturbed versus unperturbed coefficients is 
$\frac{({\bf f}, \bU_i)}{\Lambda_i}-\frac{({\bf f}, \bU_i)}{\Lambda_i+\varepsilon}=\frac{\varepsilon ({\bf f}, \bU_i)}{\Lambda_i(\Lambda_i+\varepsilon)}$, 
which is clearly of order $\varepsilon$ when $\varepsilon \ll \Lambda_i$.  
However, in our experiments this assumption does not always hold for $i=2,3$ and $\varepsilon=h$. 
Referring again to Table~\ref{table_eigs}, we see that when $c=1.25$ we have $\lambda_2 \approx0.051$.  
The convergence history for $\varepsilon=h$ and $c=1.25$ in Figure~\ref{fig1} reaches the expected asymptotic $O(h)$ when $h\ll 0.051$.
On the other hand, when $c=1.1$ we have $\lambda_2 \approx 0.0096$, 
and the convergence history in Figure~\ref{fig1} does not reach the asymptotic range because $h$ does not sufficiently resolve this small but positive eigenvalue. 
Taking $\varepsilon=h^2$ yields an algorithm that is less sensitive to small eigenvalues and thus more stable with respect to small geometric perturbations. 

\setlength{\unitlength}{.75cm}
\begin{figure}[h]
\centering
\includegraphics[scale=.4]{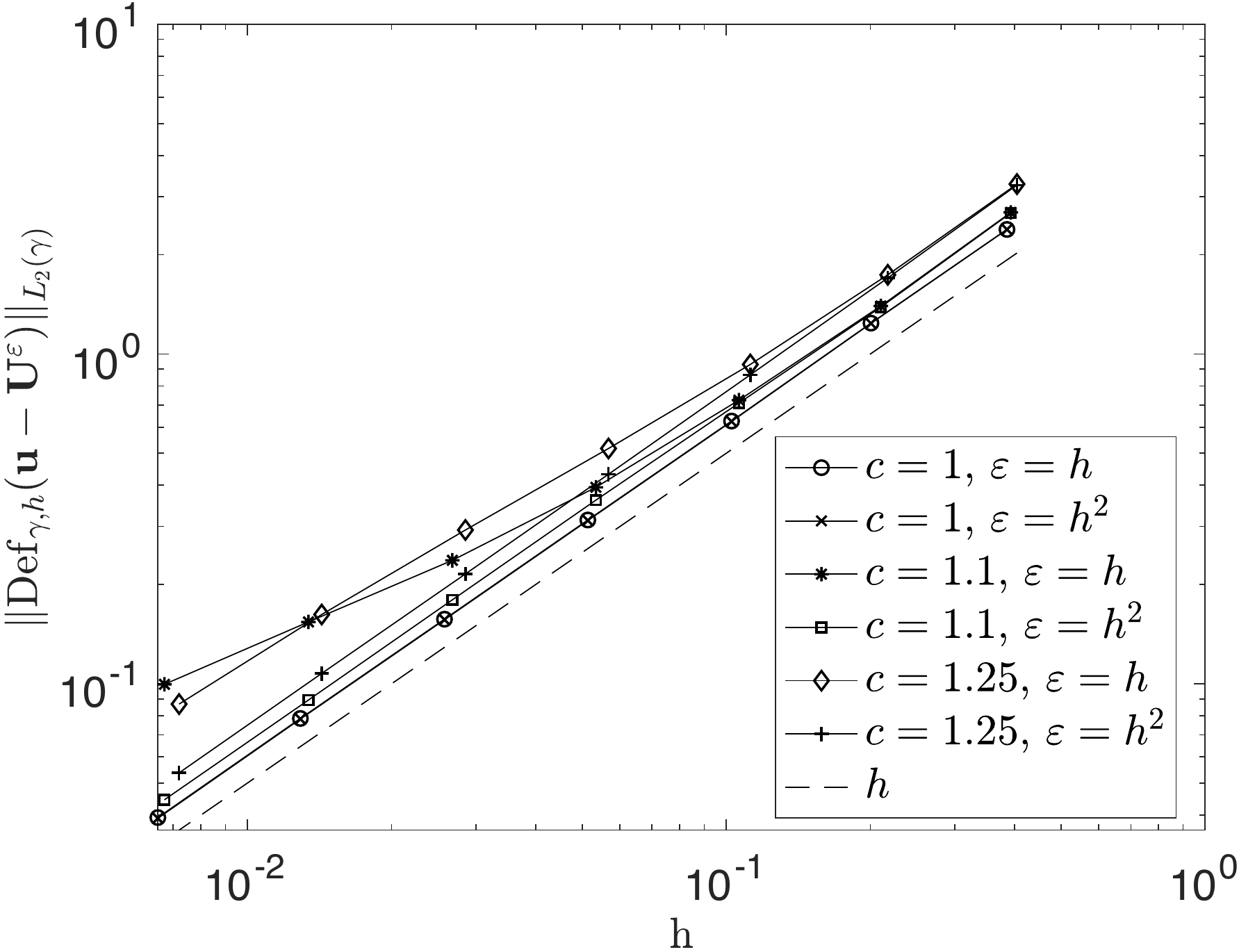}
\caption{Energy norm convergence for varying values of $c$ and $\varepsilon$.}
\label{fig1}
\end{figure}

Next we illustrate Lemma~\ref{lem:kfields} (Killing Field Estimates).  
Here we take $c=1$ so that  $\gamma$ is the unit sphere.  
In Figure~\ref{fig2}, we see that $\|P_{\mathcal{K}} \bU^\varepsilon\|_{L_2(\gamma)}=O(h)$ when $\varepsilon=h$,
and that $\|P_{\mathcal{K}} \bU^\varepsilon\|_{L_2(\gamma)}=O(1)$ when $\varepsilon =h^2$,
as predicted by Lemma~\ref{lem:kfields}.  
 
\setlength{\unitlength}{.75cm}
\begin{figure}[h]
\centering
\includegraphics[scale=.4]{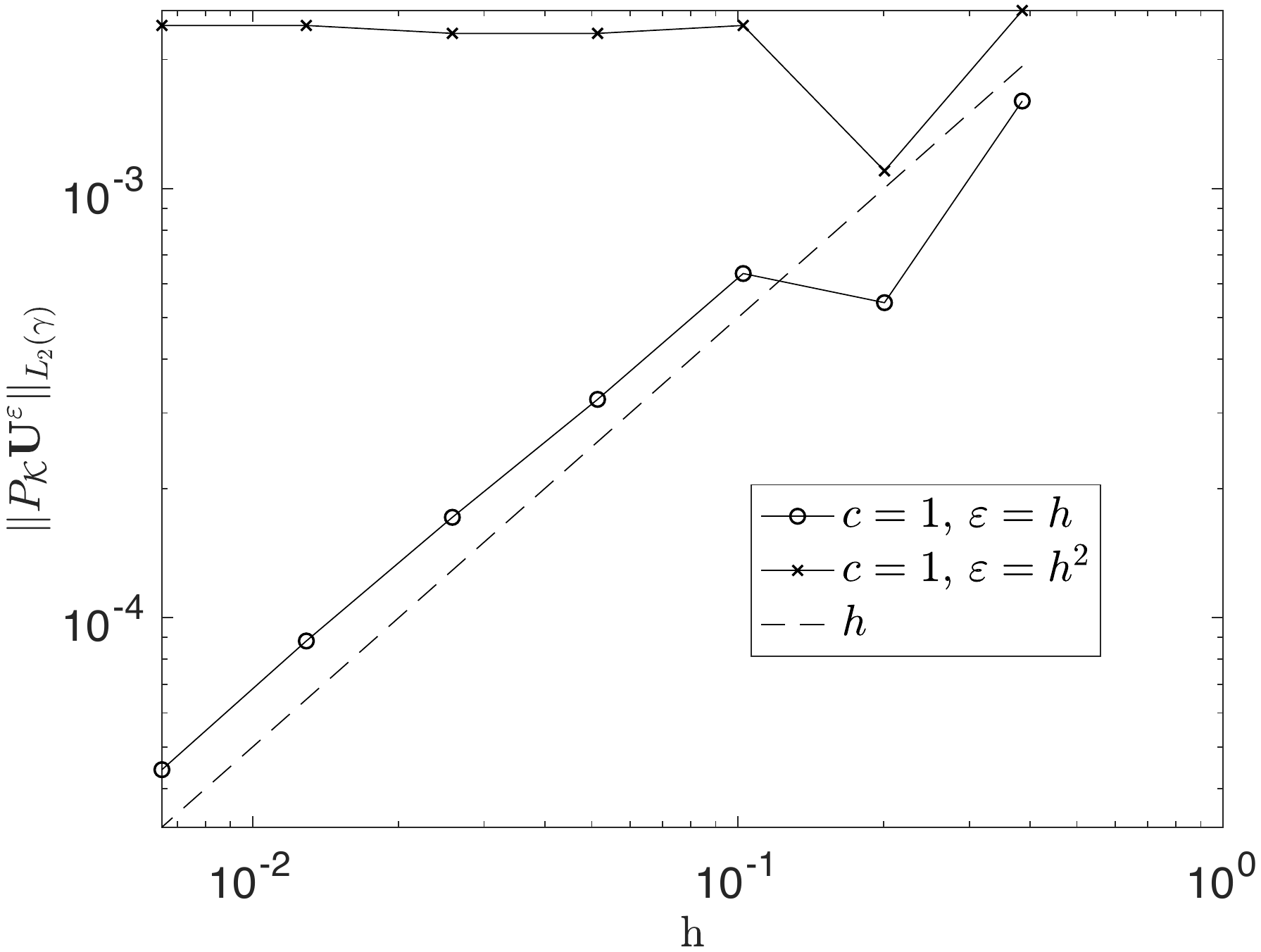}
\caption{Convergence of $\|P_{\mathcal{K}} \bU^\varepsilon\|_{L_2(\gamma)}$ for the sphere ($c=1$).}
\label{fig2}
\end{figure}

Our final experiments in this subsection illustrate the $L_2$ convergence results of Corollary~\ref{cor:L2} and Corollary~\ref{cor:L2_second}.
In the left plot of Figure~\ref{fig3} the data is given for the unit sphere ($c=1$).  
Here we clearly see that for $\varepsilon=h$, convergence in $L_2$ is of order $h$ whether or not Killing fields are filtered out.
On the other hand, when $\varepsilon=h^2$ convergence eventually stagnates if Killing fields are not removed,
but is of order $h^2$ if they are.  These results confirm the sharpness of Corollary~\ref{cor:L2} and Corollary~\ref{cor:L2_second}.
In the right plot of Figure~\ref{fig3} we consider the case $c=1.1$, where $\gamma$ is an ellipsoid which according to Table~\ref{table_eigs} has small positive Stokes eigenvalues.  
As in the case of energy norm convergence, these small eigenvalues have a substantial negative effect on the convergence behavior
when $\varepsilon=h$.  When $\varepsilon=h^2$ these small eigenvalues are seen to have some negative effect on convergence rates,
but far less than when $\varepsilon=h$.
Thus we again see that taking $\varepsilon=h^2$ yields a method that is much more stable with respect to perturbations in geometry
so long as we are able to filter out Killing fields.

\setlength{\unitlength}{.75cm}
\begin{figure}[h]
\centering
\includegraphics[scale=.34]{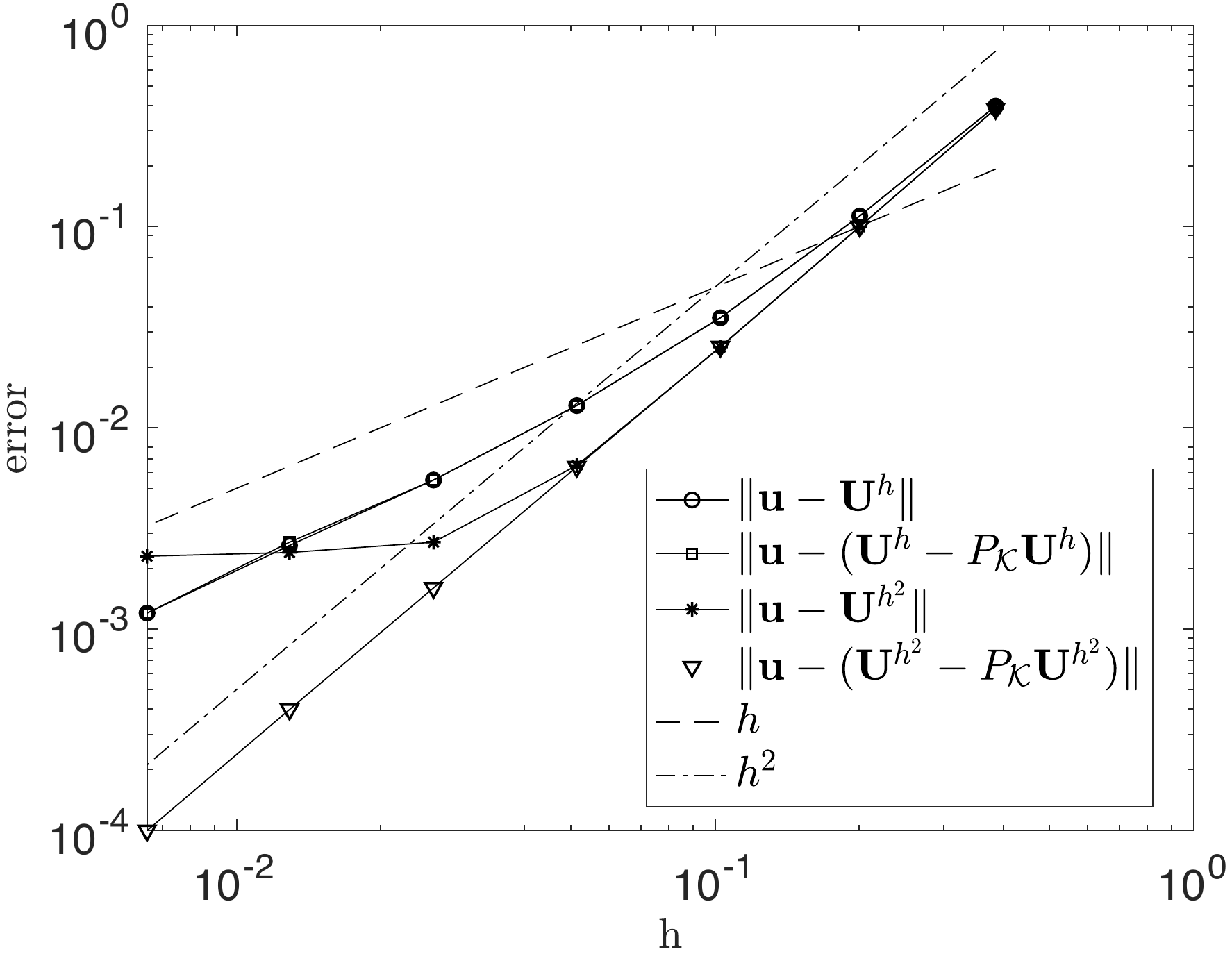}
\includegraphics[scale=.34]{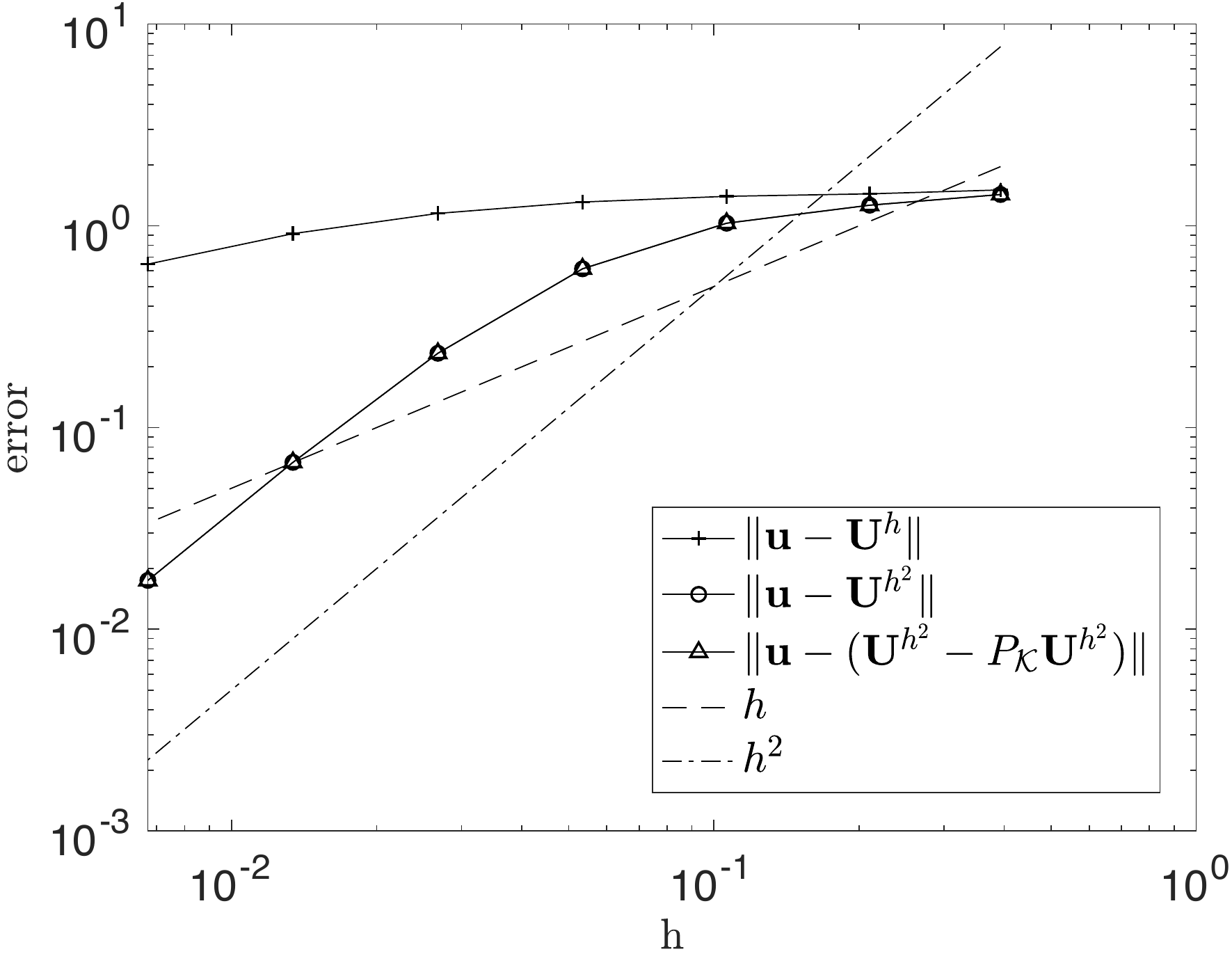}
\caption{$L_2$ convergence with and without manual filtering of Killing fields on the sphere ($c=1$, left) and on an ellipsoid with small eigenvalues ($c=1.1$, right).}
\label{fig3}
\end{figure}

\subsection{Automatic filtering of Killing fields via discrete eigenfunction computation} 
In this subsection we illustrate properties of the algorithm given in Section~\ref{subsec:Kfilt} for filtering out Killing fields. 
We focus on the ellipsoid with $c=1.1$ as it is difficult to automatically distinguish between Killing fields and modes corresponding to small positive eigenvalues in this case.  

In the left plot of Figure~\ref{fig4} we see the results obtained by comparing with $\bU^h$ ($\alpha=1$) in Corollary~\ref{cor:Kfields}.  
Here the $L_2$ error obtained from our algorithm remains nearly flat until the final data point, 
at which point it decreases sharply to mirror the error obtained by manual filtering of Killing fields.  
The reason for this relatively poor performance can be understood by referring again to Table~\ref{table_eigs}, 
where we see that the modes corresponding to $\lambda_2$ and $\lambda_3$ are viewed as Killing fields until roughly speaking $h-2h^2 \le0.0096$.  
This does not occur until the last data point in the simulation.  
These modes are thus incorrectly {\it excluded} from the computed solution in the preasymptotic range,
leading to an $O(1)$ error until the positive eigenvalues are resolved by $h-2h^2$.  
Note also that the $L_2$ error $\|\bu-\bU^h\|_{L_2(\gamma)}$ has not reached its asymptotic $O(h)$ rate here, so comparing with $\bU^h$ in the preasymptotic range is not necessarily a good strategy.  

In the right plot of Figure~\ref{fig4} we instead compare with $\bU^{h^{3/2}}$.  
Here modes are excluded from the solution only under the more stringent condition $\Lambda_i \le h^{3/2}-2h^2$.
This algorithm is thus more likely to improperly include Killing fields in the solution, 
but less likely to improperly exclude non-Killing modes.  
The results are dramatically better in the preasymptotic regime.  
Philosophically it seems preferable to include Killing fields in the solution than to improperly exclude other modes, 
so using a more stringent selection criterion is generally preferable. 
The superior performance of the choice $\alpha=3/2$ versus $\alpha=1$ in the preasymptotic range is not intuitive 
when viewed in light of Corollary~\ref{cor:Kfields}, which predicts a preasymptotic convergence rate of order $h^{2-\alpha}$.  
The discrepancy results from the previously observed phenomenon that 
small positive eigenvalues of the Stokes operator render the predicted convergence rate relatively meaningless
until the eigenvalues are sufficiently resolved by $\varepsilon$.  

\setlength{\unitlength}{.75cm}
\begin{figure}[h]
\centering
\includegraphics[scale=.34]{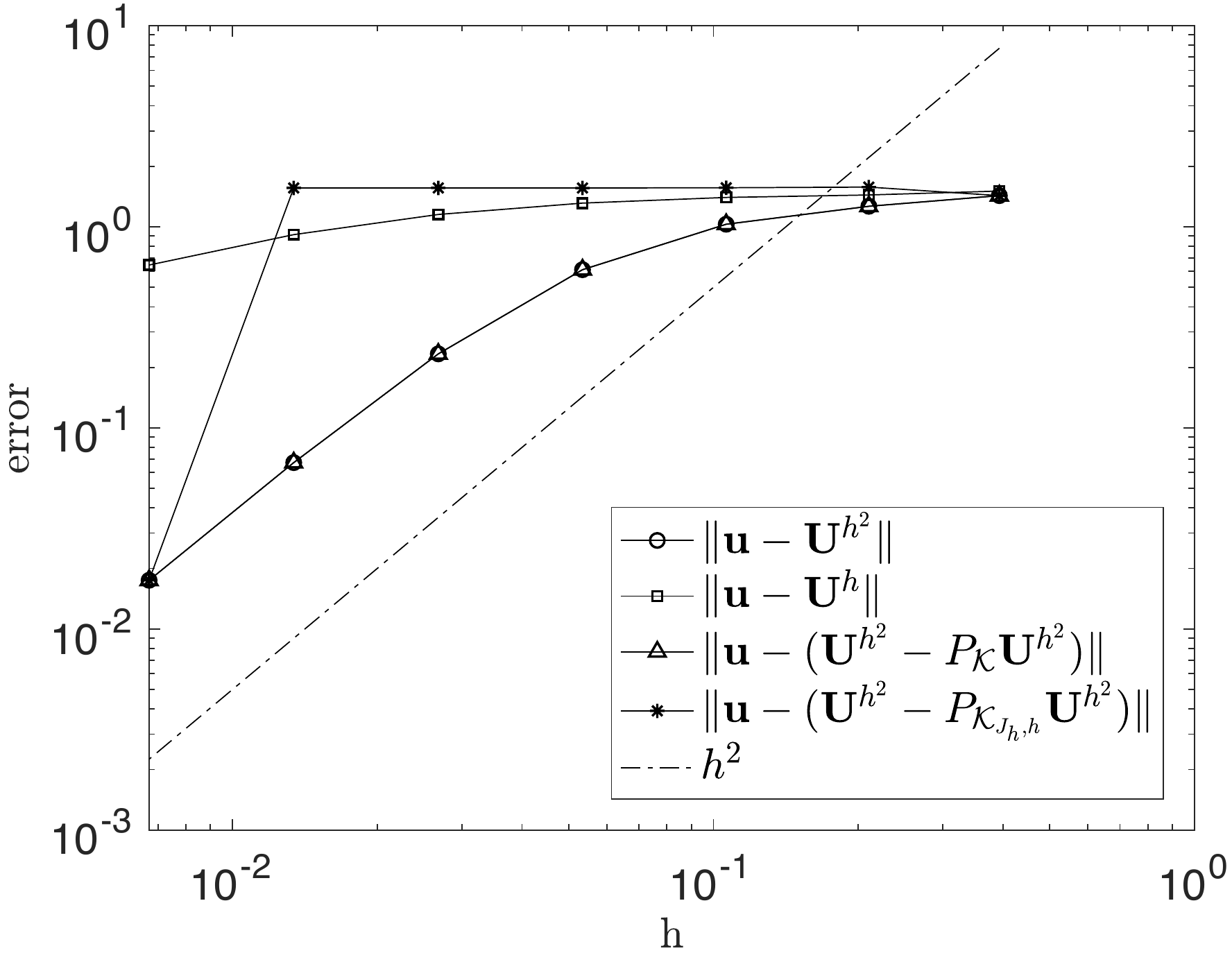}
\includegraphics[scale=.34]{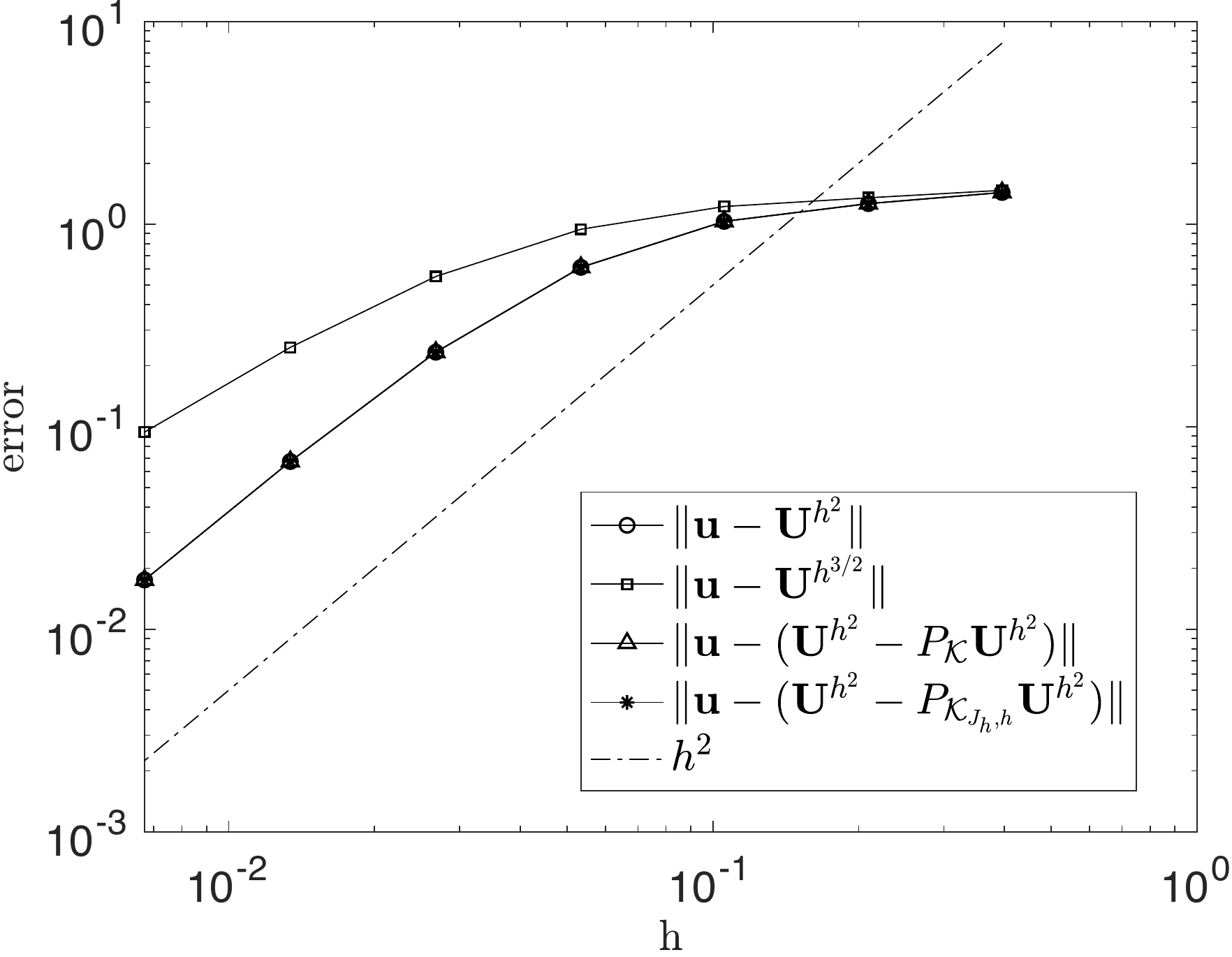}
\caption{$L_2$ convergence with automatic filtering of Killing fields by comparing with $\bU^h$ (left) and $\bU^{h^{3/2}}$ (right) on the ellipsoid with $c=1.1$.}
\label{fig4}
\end{figure}

\ARD{
\subsection{Computations when ${\bf f} \not\perp \mathcal{K}$}  In this subsection we illustrate the results of Subsection \ref{sub_kfieldsfromf}, in which options for obtaining an optimally convergent algorithm when ${\bf f} \not\perp \mathcal{K}$ are presented.   In Figure \ref{fig5} we present computational results obtained from the cases $c=2$ (with strong separation of eigenvalues corresponding to Killing fields and nondegenerate modes) and $c=1.1$ (where the smallest positive eigenvalues are relatively small).  For each case we illustrate the error estimate \eqref{l2est_pkh}, which applies when ${\rm dim}(\mathcal{K})$ is known and the discrete modes corresponding to $\mathcal{K}$ are removed from ${\bf f}$ or equivalently from ${\bf U}^\varepsilon$.  For both choices $c=1.1, 2$ we see $O(h^2)$ convergence as predicted, although in the relatively degenerate case $c=1.1$ the asymptotic rate is only seen for $h$ sufficiently small.  We also illustrate the algorithm obtained by using the criterion \eqref{eq112} in order to identify a putative discrete Killing field $\mathcal{K}_h$ when ${\rm dim}(\mathcal{K})$ is not known a priori.  For $c=2$ that the error remains relatively constant until the algorithm correctly identifies ${\rm dim}(\mathcal{K})$ and then decreases with order $h^2$.  When $c=1.1$ the condition \eqref{eq112} does not correctly identify ${\rm dim}(\mathcal{K})$ for the range of $h$ values tested, and the error remains essentially constant over the course of the calculation.  Recall that a preasymptotic convergence rate of $O(h^{2/3})$ and asymptotic rate of $O(h^2)$ is predicted for this algorithm.  As above, the theoretical preasymptotic convergence rate manifests itself in practice more as a stability guarantee rather than as a convergence rate.  However, we emphasize that without any filtering strategy the error would be expected to {\it increase} with order $h^{-2}$ when ${\bf f} \not\perp \mathcal{K}$, since the denominator $\Lambda_i+h^2$ in the eigenexpansion \eqref{eq110} is $O(h^2)$ for modes corresponding to Killing fields.    }

\setlength{\unitlength}{.75cm}
\begin{figure}[h]
\centering
\includegraphics[scale=.34]{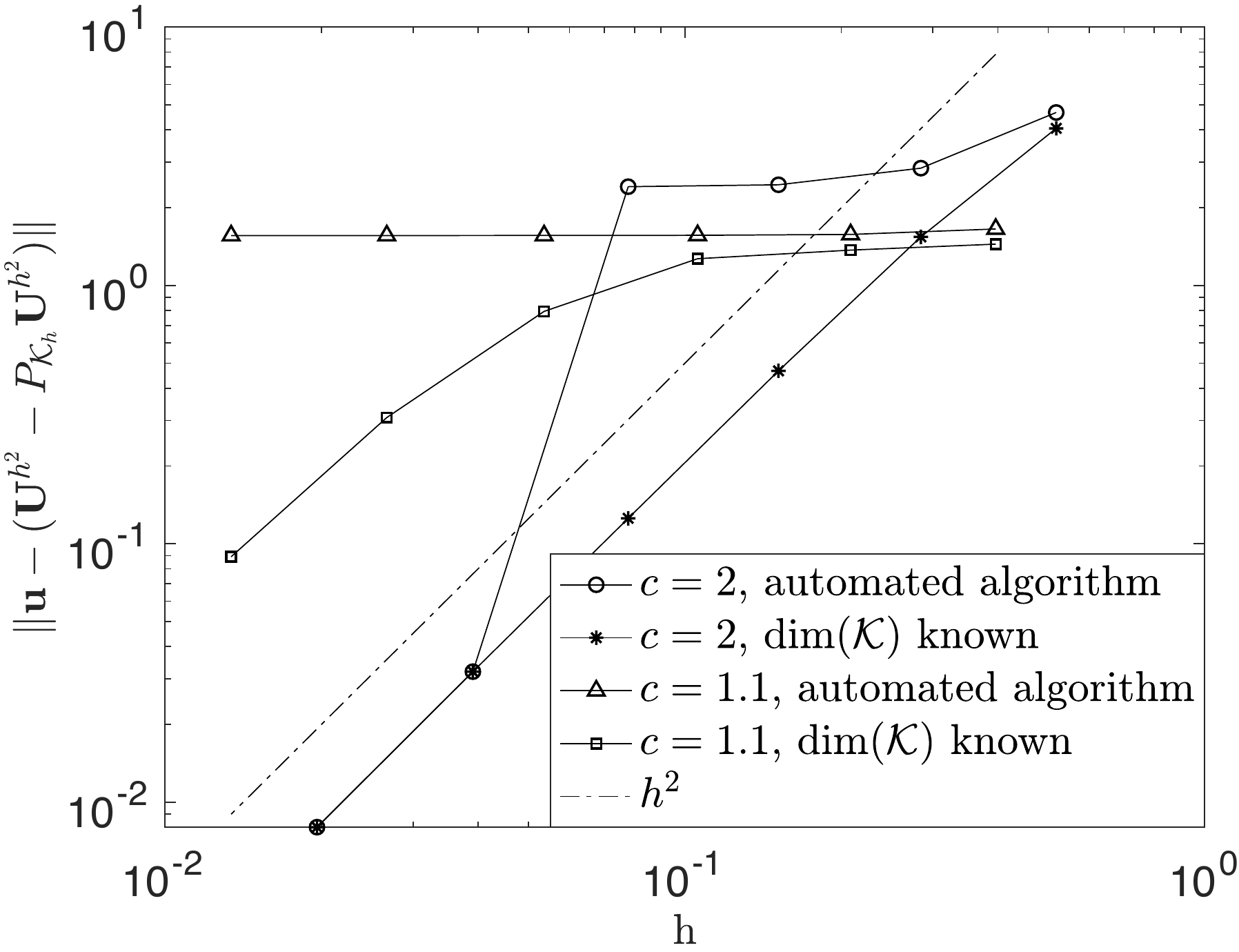}
\caption{$L_2$ convergence with automatic and manual filtering of Killing fields from ${\bf f}$.}
\label{fig5}
\end{figure}

\section{Appendix:  Discrete Korn-type inequality} \label{sec:Korn_app}
In this section we prove the discrete Korn-type Lemma~\ref{lem:disc_Korn} for the space ${\bf XT}(\T)$.  Our techniques are motivated by \cite{Bre04} with additional difficulties arising from the fact that we are working on surfaces.  

\begin{proof}
    We proceed in several steps.\\
    \noindent\boxed{1}
    We begin by defining an auxiliary discrete subspace of $H^1(\gamma)$.  Let $\Y(\bMT):=Y(\bMT)^3$, where $Y(\bMT)$ is the set of continuous piecewise linear functions on $\Gamma$, and let ${\bf YT}(\MT)=\{\Pi \overline{\bU}\circ \bP^{-1}, \overline{\bU} \in \Y(\bMT)\}$.  We then have that ${\bf YT}(\T) \subset \HT$, and we may easily derive from the continuous Korn inequality \eqref{e:Korn} that
    \begin{gather*}
    \|\bU\|_{H^1(\gamma)} \lesssim \|\Deform_\gamma(\bU)\|_{L_2(\gamma)} + \|\bU\|_{L_2(\gamma)}, ~~\bU \in {\bf YT}(\T).
    \end{gather*}
    Thus by the triangle inequality, \eqref{eq:disc_Korn} will be proved if given $\bq \in {\bf XT}(\T)$ we can find $\bv \in {\bf YT}(\T)$ such that
    \begin{align}
    \label{Korn_interp_est}
    \|\bq-\bv\|_{H_h^1(\gamma)} \lesssim  \vvvert\bq\vvvert_{1,h} + \|\bq\|_{L_2(\gamma)}.
    \end{align}

    \noindent\boxed{2} In order to define for $\bq \in {\bf XT}(\T)$ a suitable $\bv \in {\bf YT}(\T)$, first recall that $\mathcal V$ denotes the set of vertices of $\Gamma$ and $V(\bz)$ is the valence of $\bz \in \mathcal V$.  In addition, for $\bz \in \mathcal V$ let $\overline{\bf \Phi}_\bz$  be the vector $[\overline{\phi}_\bz, \overline{\phi}_\bz,\overline{\phi}_\bz]^T$ of the continuous piecewise linear hat functions with value $1$ at $\bz$ and vanishing at every other vertex. 
    We also let ${\bf \Phi}_\bz:=\overline{\bf \Phi}_\bz \circ \bP^{-1}$.  We then set
    \begin{gather*}
    \bv(\bx)= \Pi \left( \sum_{\bz \in \mathcal V} \Big (\frac{1}{V(\bz)} \sum_{\substack{T \in \T\\T \ni \bz}  } \bq_T(\bz) \Big) \star {\bf \Phi}_\bz(\bx)\right) \in {\bf YT}(\T),
    \end{gather*}
    where the multiplication $\star$ represents componentwise multiplication of vectors (given two $n$-vectors, $\bv \star \bw$ is also an $n$ vector with components given by componentwise products) and $\bq_T:= \bq|_T$ is the restriction of $\bq$ to $T$.  
    In passing,  we remark that $(A \bu) \star \Phi_\bz = A (\bu \star \Phi_\bz)$ for any $3\times 3$ matrix $A$ and any $3-$vector $\bu$.
    This property will be used repeatedly below. 
    Returning to the estimation of $\bv$, we use the definition of the infinitesimal area element \eqref{eq10} to realize that $\|{\bf \Phi}_\bz\|_{L_2(\gamma)} \lesssim h$
    so that
    \begin{gather*}
    \|\bv\|_{L_2(\gamma)}^2 \lesssim \sum_{\bz \in \mathcal V} h^2 \|\bq\|_{L_\infty(\omega_\bz)}^2 \lesssim \sum_{\bz \in \mathcal V} h^2 \|\bq \circ \bP \|_{L_\infty(\bP^{-1}(\omega_\bz))}^2,
    \end{gather*} 
    where $\omega_\bz \subset \gamma$ denotes the support of $\bf \Phi_\bz$. 
    Using this, an inverse estimate, the bounded overlap of the supports $\overline{\omega}_\bz:= \bP^{-1} (\omega_\bz)$ (see \eqref{d:shape-reg} and \eqref{d:valence}) and the norm equivalence \eqref{L2_equiv},  we find that
    \begin{gather*}
    \|\bv\|_{L_2(\gamma)}^2 \lesssim \sum_{\bz \in \mathcal V} \|\bq \circ \bP\|_{L_2(\overline{\omega}_\bz)}^2  \lesssim \|\bq \circ \bP\|_{L_2(\Gamma)}^2
    \lesssim \|\bq\|_{L_2(\gamma)}^2.
    \end{gather*}
    Thus $\|\bq-\bv\|_{L_2(\gamma)} \lesssim \|\bq\|_{L_2(\gamma)}$ and \eqref{Korn_interp_est} will follow upon proving that 
    \begin{gather*}
    \sum_{T \in \T} \|\nabla_\gamma (\bq-\bv)\|_{L_2(T)}^2 \lesssim  \vvvert\bq\vvvert^2_{1,h} + \|\bq\|_{L_2(\gamma)}^2.
    \end{gather*} 

    \noindent\boxed{3} We bound $\nabla_\gamma(\bq-\bv)$ in two steps.  Given $T \in \T$, let 
    \begin{align}
    \label{vt_def}
    \bv_T(\bx):=\Pi \sum_{\bz \in \mathcal V_T} \left(\bq_T(\bz) \star {\bf \Phi}_\bz(\bx)\right),
    \end{align}
    where $\mathcal V_T$ is the set of vertices of $T$.  First we consider the difference $\bv-\bv_T$.  
    We use $\sum_{\substack{T \in \T\\T \ni \bz}} 1=V(\bz)$ to realize that 
    \begin{gather*}
    \bv-\bv_T=\Pi \sum_{\bz \in \mathcal V_T} \Big(\frac{1}{V(\bz)} \sum_{\substack{T' \in \T\\T' \ni \bz}} (\bq_{T'}(\bz)-\bq_T(\bz)) \Big) \star {\bf \Phi}_\bz.
    \end{gather*} 
    Because $\|{\bf \Phi}_\bz\|_{H^1(T)} \lesssim 1$, we find
    \begin{gather*}
        \|\nabla_\gamma(\bv-\bv_T)\|_{L_2(T)} 
        \lesssim 
        \sup_{\bz \in \mathcal V_T} 
        \sup_{\substack{T' \in \T\\T' \ni \bz}} 
        |(\bq_{T}-\bq_{T'})(\bz)|.
    \end{gather*}
   Not all $T'$ with $z \in T'$ share an edge with $T$, but there exists a chain of adjacent elements $\{T_1,\dots,T_m\} \subset \omega_T$ with $T_1=T'$, $T_m=T$ and $m=m(\bz)$ smaller than $M$ in \eqref{d:valence} so that
    \begin{equation}\label{e:sup}
        \|\nabla_\gamma(\bv-\bv_T)\|_{L_2(T)} \lesssim \sup_{\bz \in \mathcal V_T} \sup_{i=1,\dots,m(\bz)-1} |(\bq_{T_i}-\bq_{T_{i+1}})(\bz)|.
    \end{equation}
    We now use the relation \eqref{eq10} between $\bq$ and $\barq:=\piola_{\bP^{-1}} \bq$ and the expression \eqref{eq:mu_def}  of the infinitesimal area (recalling that $\Gamma$ interpolates $\gamma$ so $d(\bz) = 0$) to deduce that
    \begin{gather*}
        \bq_{T_i}(\bz)-\bq_{T_{i+1}}(\bz)
        =
        \Pi(\bz) \Big (\frac{\barq_{\barT_i}(\bz)}{\bnu(\bz)\cdot \bnu_{\Gamma, \barT_i}}-\frac{\barq_{\barT_{i+1}}(\bz)}{\bnu(\bz)\cdot \bnu_{\Gamma, \barT_{i+1}}}    \Big ),
    \end{gather*}
    where $\overline{T}_i:= \bP^{-1}(T_i)$ and $\bnu_{\Gamma,\overline{T}_i} := \bnu_\Gamma|_{\overline{T}_i}$.
    Since 
    \begin{equation}\label{e:geom_norm}
        \frac{1}{\bnu \cdot \bnu_{\Gamma, \overline{T}_i}} -1
        = 
        \frac{1}{2 \bnu \cdot \bnu_{\Gamma, \overline{T}_i}}
        |\bnu-\bnu_{\Gamma,\overline{T}_i}|^2 
        \lesssim 
        h^2
    \end{equation}
    for $C^2$ surfaces $\gamma$,
    we deduce that
    \begin{align}
        \label{aux_comp}
        \left | \bq_{T_i}(\bz)-\bq_{T_{i+1}}(\bz) \right |  \lesssim h^2\|\barq\|_{L_\infty(\overline{\omega}_\bz)} + |\barq_{\barT_i}(\bz)-\barq_{\barT_{i+1}}(\bz)|.
    \end{align}
    The jump $\barq_{\barT_i}-\barq_{\barT_{i+1}}$ is a polynomial on the edge $\overline{e}$ shared by $\barT_i$ and $\barT_{i+1}$. Thus, an inverse inequality yields
    \begin{gather*}
        \left| \barq_{\barT_i}(\bz)-\barq_{\barT_{i+1}}(\bz)\right| 
        \lesssim 
        h^{-1/2} \| \barq_{\barT_i}-\barq_{\barT_{i+1}} \|_{L_2(\overline{e})}.
    \end{gather*}
    This, relation \eqref{eq9} between $\bq$ and $\barq$, and a computation similar to that leading to \eqref{aux_comp} guarantee that
    \begin{equation}\label{e:supb}
        \left| \barq_{\barT_i}(\bz)-\barq_{\barT_{i+1}}(\bz) \right|  
        \lesssim
        h \|\barq\|_{L_\infty(\overline{\omega}_\bz)} + h^{-1/2} \| \bq_{T_i}-\bq_{T_{i+1}} \|_{L_2(e)}.
    \end{equation}
    Returning to \eqref{e:sup} with \eqref{aux_comp} and \eqref{e:supb} at hand implies
    \begin{gather*}
        \|\nabla_\gamma(\bv-\bv_T)\|_{L_2(T)} 
        \lesssim 
        h \|\barq\|_{L_\infty(\overline{\omega}_T)}
        + 
        \sup_{\substack{e \subset \Sigma \\ e  \subset \omega_T}} 
        h^{-1/2} \|[\bq]\|_{L_2(e)},
    \end{gather*}
    where $\omega_T := \cup_{\bz \in T} \omega_\bz$. 
    The estimate for  $\nabla_\gamma (\bv-\bv_T)$, namely
    \begin{gather*}
        \sum_{T \in \T} \|\nabla_\gamma (\bv-\bv_T)\|_{L_2(T)}^2 
        \lesssim 
        h^{-1} \|[\bq]\|_{L_2(\Sigma)}^2 + \|\bq\|_{L_2(\gamma)}^2,
    \end{gather*}
    is obtained by invoking an inverse inequality 
    $h \|\barq\|_{L_\infty(\overline{\omega}_T)} \lesssim \|\barq\|_{L_2(\overline{\omega}_T)}$,
    the equivalence of norms on $T$ and $\barT$, and summing over $T \in \T$.

    \noindent\boxed{4} We now bound the term $\nabla_{\gamma} (\bv_T- \bq)$. 
    Note that since ${\bf XT} (\bMT)$ consists of piecewise linear vector fields on $\Gamma$, 
    for $\barT \subset \Gamma$ we have that 
    \begin{gather*}
        \barq_{\barT}
        =
        \sum_{\bz \in \mathcal V_T} \barq_{\barT}(\bz) \star \overline{\bf \Phi}_\bz 
        = 
        \sum_{\bz \in \mathcal V_T} \Pi_{\Gamma,\barT}\barq_{\barT}(\bz) \star \overline{\bf \Phi}_\bz,
    \end{gather*}  
    where $ \Pi_{\Gamma,\barT} :=  \Pi_{\Gamma}|_{\barT}$.
    Thus using the relation $\eqref{eq10}$ between $\bq$ and $\barq$, recalling the definition \eqref{vt_def} of $\bv_T$, and using that $\Pi-d\bH=\Pi [\bI-d\bH]$ , we find that on $T$ 
    \begin{align}
        \begin{aligned}
        \bv_T &- \bq = \Pi \sum_{\bz \in \mathcal V_T} \bq_T(\bz) \star {\bf \Phi}_\bz - \Pi \frac{1}{\mu} [\bI -d\bH] \sum_{\bz \in \mathcal V_T} \Pi_{\Gamma,\barT}\barq_{\barT}(\bz) \star {\bf \Phi}_\bz
        \\ & = \Pi \sum_{\bz \in \mathcal V_T }\left(   \frac{1}{\mu(\bz)} [ \Pi - d \bH](\bz) \barq_{\barT}(\bz)\star {\bf \Phi}_\bz- \frac{1}{\mu} [\bI -d\bH]   \Pi_{\Gamma,\barT}\barq_{\barT}(\bz) \star {\bf \Phi}_\bz \right)
        \\ & = \Pi \sum_{\bz \in \mathcal V_T}\left( \frac{1}{\bnu(\bz) \cdot \bnu_{\Gamma, \barT}} \Pi(\bz) \barq_{\barT}(\bz) \star {\bf  \Phi}_\bz- \frac{1}{\mu} [\bI -d\bH]  \Pi_{\Gamma,\barT} \barq_{\barT}(\bz) \star {\bf \Phi}_\bz\right).
        \end{aligned}
    \end{align}
    Here we have also taken advantage of the fact $\bz \in \Gamma \cap \gamma$ to use $d(\bz) = 0$ and so $\mu(\bz)=\bnu\cdot \bnu_\Gamma$. 
    Recall that $\nabla_\gamma (\bq-\bv_T) = \Pi \nabla (\bq-\bv_T)\Pi$,
    where abusing notation we have denoted by $\bq$ the extension of $\bq:\gamma \rightarrow \mathbb R^3$ to $\mathcal N$ using the distance lift, i.e. $\bq(\bx) = \bq(\bP(\bx))$.
    Using the product rule, we set
    $\nabla_\gamma (\bq-\bv_T) = \Pi(I+II+III)\Pi$,
    where
    \begin{align*}
        I&:= \nabla \Pi \sum_{\bz \in \mathcal V_T}\left( \frac{1}{\bnu(\bz) \cdot \bnu_{\Gamma, \barT}} \Pi(\bz) \barq_{\barT}(\bz) \star {\bf  \Phi}_\bz- \frac{1}{\mu} [\bI -d\bH]  \Pi_{\Gamma,\barT} \barq_{\barT}(\bz) \star {\bf \Phi}_\bz\right),\\
        II&:=  -\Pi \nabla \left( \frac{1}{\mu} [\bI -d\bH] \right) \sum_{\bz \in \mathcal V_T} \Pi_{\Gamma,\barT} \barq_{\barT}(\bz) \star {\bf \Phi}_\bz\\
        III&:=  \Pi\sum_{\bz \in \mathcal V_T} \left( \frac{1}{\bnu(\bz) \cdot \bnu_{\Gamma, \barT}} \Pi(\bz)- \frac{1}{\mu} [\bI -d\bH] \Pi_{\Gamma,\barT} \right) \nabla\left(\barq_{\barT}(\bz) \star  {\bf  \Phi}_\bz\right).
    \end{align*}
    Before estimating each term, we recall that the assumed regularity on $\gamma$ guarantees that 
    \begin{equation}\label{e:tmp_geom}
        \| \Pi \|_{W^1_\infty(\mathcal N)} + \| \Pi_{\Gamma} \|_{W^1_\infty(\Gamma)} + \| \bH \|_{W^1_\infty(\mathcal N)} \lesssim 1
    \end{equation}
    and $\|d \|_{L_\infty(\Gamma)} + \|1-\mu^{-1}\|_{L_\infty(\mathcal N)}  \lesssim h^2$, which together with $d \in C^3(\mathcal{N})$ implies
    \begin{equation}\label{e:tmp_geom2}
        \| \mu^{-1} (\bI-d\bH)-\bI\|_{L_\infty(\Gamma)} + h^2 \| \nabla(\mu^{-1} (\bI-d\bH))\|_{L_\infty(\Gamma)} \lesssim h^2.
    \end{equation}
    In addition, $\bnu=\nabla d$ is $C^1$ and $|\bnu-\bnu_\Gamma| \lesssim h$, so
    \begin{equation}\label{e:tmp_geom3}
        \|\Pi-\Pi(\bz)\|_{L_\infty(T)} + |\Pi-\Pi_\Gamma| \lesssim h \Longrightarrow  |\Pi(\bz)-\Pi_{\Gamma,\barT}| \lesssim h.
    \end{equation}

    We can now start with $I$ and invoke \eqref{e:tmp_geom}, \eqref{e:tmp_geom2} together with \eqref{e:geom_norm} to write
    \begin{gather*}
        \|I\|_{L_\infty(T)} \lesssim h^2|\barq_{\barT}(\bz)| + | (\Pi(\bz)  - \Pi_{\Gamma,\barT})\barq_{\barT}(\bz)|.
    \end{gather*}
    In view of \eqref{e:tmp_geom3}, this leads to
    \begin{gather*}
        \|I\|_{L_\infty(T)} \lesssim h | \barq_{\barT}(\bz)|.
    \end{gather*}
    Thus, an inverse inequality and the norm equivalence  \eqref{L2_equiv} shows that
    \begin{gather*}
        \|I\|_{L_2(T)}  \lesssim h\|I\|_{L_\infty(T)} \lesssim h \|\barq\|_{L_\infty(\barT)} \lesssim \|\barq\|_{L_2(\barT)}.
    \end{gather*}

    To estimate $II$, we take again advantage of \eqref{e:tmp_geom} and \eqref{e:tmp_geom2} to obtain
    \begin{gather*}
        \|II\|_{L_2(T)} \lesssim h\|II\|_{L_\infty(T)}\lesssim h\|\barq\|_{L_\infty(\barT)} \lesssim \|\bq\|_{L_2(T)}.
    \end{gather*}
    Finally, for $III$, we estimate (up to a multiplicative constant) $\| III \|_{L_2(T)}$ by
    \begin{gather*}
        \max_{\bz \in \mathcal V_T} \left \|\Pi \Big ( \frac{1}{\bnu(\bz)\cdot \bnu_{\Gamma, \barT}} \Pi(\bz)-\frac{1}{\mu}[\bI-d \bH] \Big ) \Pi_\Gamma \right\|_{L_\infty(\barT)} \|\barq \|_{L_\infty(\barT)} \|\nabla {\bf \Phi}_\bz\|_{L_2(T)}.
    \end{gather*}
    Using \eqref{e:tmp_geom2}, \eqref{e:geom_norm}, $\|\nabla {\bf \Phi}_\bz\|_{L_2(T)} \lesssim 1$ and \eqref{e:tmp_geom3}, we arrive at
    \begin{gather*}
        \|III\|_{L_2(T)} \lesssim h \|\barq\|_{L_\infty(\barT)} \lesssim \|\bq\|_{L_2(T)}.
    \end{gather*}
    Gathering the above estimates for $I$, $II$, and $III$ yields
    \begin{gather*}
        \|\nabla_\gamma (\bq-\bv_T)\|_{L_2(T)} \lesssim \|\bq\|_{L_2(T)},
    \end{gather*}
    which completes the proof.
\end{proof}

\bibliographystyle{siam}

\end{document}